\documentclass[12 pt]{article}
\usepackage[pdftex]{graphicx}
\usepackage{amsmath}
\usepackage{mathtools}
\usepackage{amsthm}
\usepackage{amsfonts}
\usepackage{amssymb}
\usepackage[margin=1.0in]{geometry}
\usepackage{tikz-cd}
\usetikzlibrary{cd}
\usepackage{enumitem}

\newtheorem{theorem}{Theorem}[section]
\newtheorem{cor}[theorem]{Corollary}
\newtheorem{lemma}[theorem]{Lemma}
\newtheorem{prop}[theorem]{Proposition}

\newtheorem*{theorem*}{Theorem}

\theoremstyle{definition}
\newtheorem{defin}[theorem]{Definition}
\newtheorem{fact}[theorem]{Fact}
\newtheorem{exa}[theorem]{Example}
\newtheorem{que}[theorem]{Question}
\newtheorem{notation}[theorem]{Notation}

\theoremstyle{remark}
\newtheorem*{rem}{Remark}
\newtheorem*{claim}{Claim}

\renewcommand{\rm}[1]{\mathrm{#1}}

\newcommand{\cA}{\mathcal{A}}

\newcommand{\rmB}{\mathrm{B}}
\newcommand{\cB}{\mathcal{B}}

\newcommand{\bbC}{\mathbb{C}}
\newcommand{\rmC}{\mathrm{C}}

\newcommand{\rmD}{\mathrm{D}}

\newcommand{\rmF}{\mathrm{F}}
\newcommand{\cF}{\mathcal{F}}

\newcommand{\sfG}{\mathsf{G}}

\newcommand{\bK}{\mathbf{K}}

\newcommand{\cL}{\mathcal{L}}

\newcommand{\bM}{\mathbf{M}}

\newcommand{\rmM}{\mathrm{M}}

\newcommand{\bN}{\mathbf{N}}
\newcommand{\bbN}{\mathbb{N}}
\newcommand{\rmN}{\mathrm{N}}
\newcommand{\cN}{\mathcal{N}}

\newcommand{\cO}{\mathcal{O}}

\newcommand{\cP}{\mathcal{P}}

\newcommand{\sfp}{\mathsf{p}}

\newcommand{\cQ}{\mathcal{Q}}

\newcommand{\sfq}{\mathsf{q}}

\newcommand{\bbR}{\mathbb{R}}

\newcommand{\rmS}{\mathrm{S}}

\newcommand{\cU}{\mathcal{U}}

\newcommand{\cV}{\mathcal{V}}

\newcommand{\cW}{\mathcal{W}}

\newcommand{\bbZ}{\mathbb{Z}}

\newcommand{\wt}{\widetilde}
\newcommand{\wh}{\widehat}
\newcommand{\ol}{\overline}

\newcommand{\Sa}{\mathrm{Sa}}
\newcommand{\sa}{\mathrm{Sa}}
\newcommand{\la}{\langle}
\newcommand{\ra}{\rangle}

\newcommand{\fr}{Fra\"iss\'e }
\renewcommand{\phi}{\varphi}

\newcommand{\aut}{\mathrm{Aut}}

\newcommand{\dom}{\mathrm{dom}}

\newcommand{\im}{\mathrm{Im}}

\newcommand{\str}{\mathbf{Str}}
\newcommand{\Int}{\mathrm{Int}}

\newcommand{\sn}{\rm{SN}}
\newcommand{\snrpc}{\rm{SN}_{\rm{RPC}}}
\newcommand{\snpc}{\rm{SN}_{\rm{PC}}}
\newcommand{\op}{\rm{op}}
\newcommand{\ngrpc}{\cN_G^{\rm{RPC}}}

\newcommand{\subg}{\rm{Sub}_G}

\newcommand{\fubini}[3]{\Sigma_{#1}^{#3}\Sigma_{#2}^{#3}}
\newcommand{\fin}[1]{[#1]^{<\omega}}

\def\-{\raisebox{.30pt}{-}}

\begin{document}
	
	\title{Ultracoproducts and weak containment for flows of topological groups}
	\author{Andy Zucker}
	\date{January 2024}
	\maketitle
	
	\begin{abstract}
		We develop the theory of ultracoproducts and weak containment for flows of arbitrary topological groups. This provides a nice complement to corresponding theories for p.m.p.\ actions and unitary representations of locally compact groups. For the class of locally Roelcke precompact groups, the theory is especially rich, allowing us to define for certain families of $G$-flows a suitable compact space of weak types. When $G$ is locally compact, all $G$-flows belong to one such family, yielding a single compact space describing all weak types of $G$-flows.
		\let\thefootnote\relax
		\let\thefootnote\relax
		\let\thefootnote\relax\footnote{The author was supported by NSERC grants RGPIN-2023-03269 and DGECR-2023-00412.}
	\end{abstract}
	
	\section{Introduction}
	
	The notions of weak containment and weak equivalence have been of great importance in the study of representation theory and ergodic theory of locally compact groups. First defined for unitary representations by Godemont \cite{Godem_48} and developed further by Fell \cite{Fell_60, Fell_62}, weak containment was later defined by Kechris \cite{Kechris_Global} for probability measure preserving actions of locally compact groups on probability spaces. In both settings, a variety of useful properties a representation or a p.m.p.\ action might enjoy are invariants of weak equivalence, while at the same time being a coarse enough relation to be tractable, as opposed to simply considering isomorphism. We refer to the survey \cite{Burton_Kechris_wk} for more on weak containment of p.m.p. actions of countable groups.
	
	Very often, notions of weak containment can equivalently be phrased by developing a suitable notion of ultraproduct for the given class of objects; then one object weakly contains another exactly when an ultrapower of the first object embeds the second. Ultraproducts were first defined for first-order structures; Dacunha-Castelle and Krivine in \cite{DCK_BSpace_Ult} defined a notion of ultraproduct for Banach spaces, which was further developed by Henson \cite{Henson_BSpace_Ult}, thus allowing for a notion of ultraproduct for unitary representations of discrete groups. Loeb \cite{Loeb} defined a notion of ultraproduct for probability spaces, which was later used by Conley, Kechris and Tucker-Drob \cite{CKTD_Ults} to develop a notion of ultraproduct for p.m.p.\ actions of countable groups. More recently, Ben Yaacov and Goldbring in \cite{BYG_2021} defined two different notions of ultraproduct for unitary representations of locally compact groups, representing two different approaches to dealing with the non-discrete group topology.

	In this paper, we develop the theory of weak containment for flows of topological groups, i.e.\ a compact space $X$ equipped with a continuous action $G\times X\to X$, where $G$ is a topological group. We do this by first developing the theory of \emph{$G$-equicontinuous ultracoproducts}, then \emph{defining} weak containment via the property that a flow weakly contains another iff an ultracopower of the former factors onto the latter. The ultracoproduct construction for compact Hausdorff spaces was thoroughly developed by Bankston in a series of works \cite{Bankston_1984, Bankston_1987, Bankston_2003}, and the ultracoproduct construction of $G$-flows is implicit in work of Schneider \cite{Schneider_Equivariant} connecting topological dynamics to Gromov's metric measure geometry. When the group $G$ is allowed to be any topological group, a major difficulty of working at this level of generality is that the theory of weak containment becomes extremely subtle. For certain topological groups, the \emph{Fubini groups} that we define in Section~\ref{Section:Fubini}, we show that on a large class of $G$-flows, weak containment is indeed a pre-order. For \emph{locally Roelcke precompact} groups, we give a combinatorial description of the \emph{weak type} of a $G$-flow, which for many $G$-flows precisely captures its weak equivalence class.
	
	Two applications of this theory were major motivations for the work contained here and will appear in future works which we briefly preview. First, in upcoming joint work with G.\ Basso, we give a new characterization of those Polish groups with the property that their universal minimal flow has a comeager orbit. A major component of the proof is the analysis of ultracopowers of the universal minimal flow. We show that if $G$ is a Polish group whose universal minimal flow is non-metrizable and has a comeager orbit, then $\rmM(G)$ is ``almost" weakly rigid (Definition~\ref{Def:Weak_Rigid}) in that any ultracopower cannot be too much larger.
	
	The other application pertains to connections between topological dynamics and the notion of \emph{big Ramsey degrees} from structural Ramsey theory. The seminal paper of Kechris, Pestov, and Todor\v{c}evi\'c \cite{KPT} connects the study of the universal minimal flow of $\aut(\bK)$ for a countable ultrahomogeneous first-order structure $\bK$ (these are often called \emph{\fr structures}) to a property of the class of finite structures which embed into $\bK$ called the \emph{Ramsey property}. In \cite{Zucker_Metr_UMF}, the present author shows that for $\aut(\bK)$ as above, the metrizability of the universal minimal flow is exactly characterized by the associated class of finite structures having finite \emph{small Ramsey degrees}. In \cite{Zucker_BR_Dynamics}, a new dynamical object called the \emph{universal completion flow} is defined, and assuming a mild strengthening of finite big Ramsey degrees, it is shown that $\aut(\bK)$ admits a universal completion flow which is metrizable and unique. However, lacking from this result was a uniquely defined dynamical object that exists for any topological group, regardless of if $G$ has the form $\aut(\bK)$ or if $\bK$ has finite big Ramsey degrees. In upcoming work, such a dynamical object is given. The key difficulty is that while this dynamical object is not unique up to isomorphism, it is unique up to weak equivalence. 
	\vspace{3 mm}
	
	\textbf{Acknowledgments:} I thank Gianluca Basso for numerous detailed discussions as this project unfolded. I also thank Dana Barto\v{s}ov\'a for pointing me towards the work of Bankston, and I thank Isaac Goldbring and Martin Schneider for comments on an earlier draft.    
	
	\subsection{Notation and conventions}
	
	Most set-theoretic notation is standard. We write $\bbN = \omega$ for the set of non-negative integers, write $k< \omega$ when $k\in \omega$, and given $k< \omega$, we identify $k$ with the set $\{0,..., k-1\}$. If $I$ is a set, we write $\la a_i: i\in I\ra$ for the function with domain $I$ which sends $i\in I$ to $a_i$. We often call functions introduced this way \emph{tuples}.
	
	Most model-theoretic notation is also standard. Unless otherwise specified, we typically denote first-order structures with bold letters and let the un-bolded version denote the underlying set, i.e.\ $\bM$ has underlying set $M$.
	
	All groups and spaces in these notes are Hausdorff. If $X$ is a topological space, we let $\op(X)$ denote the set of non-empty open subsets of $X$ and $\exp(X)$ denotes the set of non-empty closed subsets of $X$. Given $x\in X$, we write $\op(x, X):= \{A\in \op(X): x\in A\}$. We let $\rmC(X)$ denote the algebra of continuous bounded functions from $X$ to $\bbC$. If $s > 0$, we put $\rmC^s(X) = \{f\in \rmC(X): \|f\|\leq s\}$. We write $\rmC(X, [0,1])$ for the continuous functions from $X$ to $[0, 1]$.
	
	We will take all pseudo-metrics to be bounded. If $\rho$ is a pseudometric on some set $X$, we take $\rho$-Lipschitz to refer to Lipschitz constant $1$; to refer to  another Lipschitz constant $c> 0$, we can form the pseudo-metric $c\rho$ and refer to $c\rho$-Lipschitz functions.

	\section{Ultracoproducts of compact spaces}
	\label{Section:Ultracoproducts_Spaces}
	
	Ultracoproducts and ultraproducts of families of compact spaces have been investigated by Bankston\footnote{I thank Dana Barto\v{s}ov\'a for pointing me to the reference \cite{Bankston_2003}.} (see \cite{Bankston_2003} and the references therein). In our construction of an ultracoproduct of $G$-flows, we will build the underlying space in a similar fashion, and in the case that $G$ is a discrete group, the underlying space of the ultracoproduct is exactly the ultracoproduct of the underlying spaces.
	
	We will make heavy use of \emph{Gelfand duality}, which states that the categories of compact Hausdorff spaces and unital commutative $C^*$-algebras are contravariantly isomorphic. To each compact Hausdorff space $X$, one associates the algebra $\rmC(X)$, and given a unital commutative $C^*$-algebra $\cA$, one forms the \emph{Gelfand space} of $\cA$, the space $\wh{\cA}$ of multiplicative linear functionals $\cA\to \bbC$ equipped with the topology of pointwise convergence. If $X$ and $Y$ are compact and $\phi\colon X\to Y$ is continuous, then one obtains a $*$-homomorphism $\hat{\phi}\colon \rmC(Y)\to \rmC(X)$ via $\hat{\phi}(f) = f\circ \phi$. Conversely, if $\cA$ and $\cB$ are unital commutative $C^*$-algebras and  $\eta\colon \cA\to \cB$ is a $*$-homomorphism, then one obtains a continous map $\hat{\eta}\colon \wh{\cB}\to \wh{\cA}$ via $\hat{\eta}(x)(a) = x(\eta(a))$.    
	
	Let $I$ be an infinite set and $\vec{X}= \la X_i: i\in I\ra$ a tuple of compact spaces. Then $\bigsqcup\vec{X}:= \bigsqcup_{i\in I} X_i$ is locally compact, and as such we can naturally view it as a dense open subspace of its \emph{beta compactification} $\beta(\bigsqcup \vec{X}):= \wh{\rmC(\bigsqcup\vec{X})}$, where we note that $\rmC(\bigsqcup \vec{X})\cong \bigcup_{s>0} \prod_{i\in I} \rmC^s(X_i)$. This space by definition satisfies the following universal property: for any compact space $Y$ and any continuous map $\phi\colon \bigsqcup \vec{X}\to Y$, there is a continuous extension $\wt{\phi}\colon \beta(\bigsqcup \vec{X})\to Y$. In particular, considering the continuous map $\pi_I\colon \bigsqcup \vec{X}\to \beta I$ with $\pi_I(x) = i$ iff $x\in X_i$, we obtain a continuous extension to $\beta(\bigsqcup \vec{X})$, which we also denote by $\pi_I$. Given an ultrafilter $\cU\in \beta I$, the \emph{ultracoproduct} of $\la X_i: i\in I\ra$ along $\cU$ is the space $\Sigma_\cU X_i := \pi_I^{-1}(\{\cU\})$. We can identify $\rmC(\Sigma_\cU X_i)$ with the  ultra\emph{product} of the $C^*$-algebras $\{\rmC(X_i)\colon i\in I\}$ along $\cU$. This is the algebra $$\left(\bigcup_{s>0} \prod_{i\in I} \rmC^s(X_i)\right)/ \sim_\cU,$$ where given $(p_i)_{i\in I}, (q_i)_{i\in I}\in \prod_{i\in I} \rmC^s(X_i)$, we declare that $(p_i)_{i\in I}\sim_\cU (q_i)_{i\in I}$ iff for every $\delta > 0$, we have $\{i\in I: \|p_i-q_i\|< \delta\}\in \cU$. Addition and multiplication are then defined coordinate-wise, and we set $\|[(p_i)_{i\in I}]_{\sim_\cU}\| = \lim_{i\to \cU} \|p_i\|$. We write $(p_i)_\cU\in \rmC(\Sigma_\cU X_i)$ for the corresponding continuous function on the ultracoproduct.
	
	When $X_i\cong X$ for every $i\in I$, we call $\Sigma_\cU X$ the \emph{ultracopower} of $X$ along $\cU$. In this case, the projection map $I\times X\to X$ continuously extends to $\beta(I\times X)$, and we let $\pi_{X, \cU}$ denote its restriction to $\Sigma_\cU X$; we call this the \emph{ultracopower map}.

	We now turn to ultraproducts. Recall that the ultraproduct of the \emph{sets} $X_i$ along $\cU$ is defined by $\Pi_\cU X_i:= \prod_{i\in I} X_i/E_\cU$, where $(x_i)_{i\in I} E_\cU (y_i)_{i\in I}$ iff $\{i\in I: x_i = y_i\}\in \cU$. Viewing the $X_i$ as spaces, consider the map $\lim_{\vec{X},\cU}\colon \prod_{i\in I} X_i\to \Sigma_\cU X_i$ given by $\lim_{\vec{X},\cU}((x_i)_{i\in I}) = \lim_{i\to \cU} x_i$. This map is $E_\cU$-invariant, giving us a map $\iota_{\vec{X},\cU}\colon \Pi_\cU X_i\to \Sigma_\cU X_i$.
	\vspace{3 mm} 
	
	\begin{claim}
		$\iota_{\vec{X},\cU}$ is injective.
	\end{claim}
	
	\begin{proof}
		Suppose $(x_i)_{i\in I}$ and $(y_i)_{i\in I}$ satisfy $I':=\{i\in I: x_i\neq y_i\}\in \cU$. Find $f\in \rmC(\bigsqcup \vec{X})$ so that $f(x_i) = 0$ and $f(y_i) = 1$ for each $i\in I'$. Upon continuously extending $f$, we must have $f\circ\lim_{\vec{X},\cU}((x_i)_{i\in I}) = 0$ and $f\circ \lim_{\vec{X},\cU}((y_i)_{i\in I}) = 1$, and hence $\lim_{\vec{X},\cU}((x_i)_{i\in I})\neq \lim_{\vec{X},\cU}((y_i)_{i\in I})$.
	\end{proof}
	
	It turns out (see \cite{Bankston_1987}) that the ultraproduct of the spaces $\la X_i: i\in I\ra$ is exactly the topology on $\Pi_\cU X_i$ which turns $\iota_{\vec{X},\cU}$ into a homeomorphism. Hence we suppress the notation $\iota_{\vec{X},\cU}$ and identify $\Pi_\cU X_i$ as a subspace of $\Sigma_\cU X_i$. 
	\vspace{3 mm}
	
	\begin{claim}
		If for every $n< \omega$ we have $\{i\in I: |X_i|> n\}\in \cU$, then $\Pi_\cU X_i\subsetneq \Sigma_\cU X_i$.
	\end{claim}

	\begin{proof}
		Fix for each $\vec{x} = (x_i)_{i\in I}\in \prod_{i\in I} X_i$ a tuple of functions $(p_i^{\vec{x}})_{i\in I}$ with each $p_i^{\vec{x}}\in \rmC(X_i, [0, 1])$, with $p_i^{\vec{x}}(x_i) = 0$, and with the property that for any $Q\in \fin{\prod_{i\in I} X_I}$, we have $\lim_\cU \|\prod_{\vec{x}\in Q} p_i^{\vec{x}}\| = 1$. Then there is $y\in \Sigma_\cU X_i$ with the property that $(p_i^{\vec{x}})_\cU(y) = 1$ for every $\vec{x}\in \prod_{i\in I} X_i$, which cannot hold if $y\in \Sigma_\cU X_i$.   
	\end{proof}

	In the case of ultra(co)powers, we have the (set-theoretic) ultrapower embedding $\delta_{X,\cU}\colon X\to \Pi_\cU X$, so also a map $j_{X,\cU}:= \iota_{X,\cU}\circ \delta_{X,\cU}\colon X\to \Sigma_\cU X$. In general $j_{X,\cU}$ is injective, but not continuous. 
	
	If $\{G_i: i\in I\}$ are discrete groups and each $X_i$ is a $G_i$-flow, we obtain an action of $\Pi_{i\in I} G_i$ on $\beta(\bigsqcup \vec{X})$ by simply taking the continuous extensions which are guaranteed to exist. In this way, we obtain an action of the ultraproduct $\Pi_\cU G_i$ on $\Sigma_\cU X_i$. When each $G_i$ is the same discrete group $G$ and we identify $G$ via it's image in $\Pi_\cU G$ under the ultrapower embedding, this is exactly the ultracoproduct of the $G$-flows $\la X_i: i\in I\ra$. 
	
	The discrete case is already quite useful for discussing some examples. 
	
	\begin{exa}
		Let us show that the ultracopower of a minimal flow need not be minimal. Consider $G = \bbZ$, and let $X$ be an irrational rotation of the circle. Fix $\cU\in \beta \omega\setminus \omega$, and form $\Sigma_\cU X\subseteq \beta (X\times \omega)$. Fix $x\in X$, and let $f_n\colon X\to [0,1]$ be continuous functions with $f_n(x) = 1$ and $f_n(y) = 0$ for any $y\in X\setminus (x-1/n, x+1/n)$. Let $f\colon \Sigma_\cU X\to [0,1]$ be the continuous function represented by $(f_n)_{n< \omega}$, and consider the open set $A:= \{y\in \Sigma_\cU X: f(y)> 1/2\}$. Then $j_{X,\cU}(x)\in A$, but $g\cdot j_{X,\cU}(x)\not\in A$ for any $g\in \bbZ\setminus \{0\}$.
		
		By contrast, certain topological groups will have the property that all ultracoproducts of minimal flows remain minimal; see Corollary~\ref{Cor:CAP_Ults}.
	\end{exa}

\section{$G$-continuity and $G$-compactification}
\label{Section:G_Cont}

For the rest of the paper, $G$ denotes an arbitrary topological group, and $G_{dsc}$ denotes $G$ with the discrete topology. We let $e_G$ denote the identity element and $\cN(G)$ denote a fixed base of symmetric open neighborhoods of $e_G$. We write $\rm{FS}(G)$ for the set of finite symmetric subsets of $G$ containing the identity, and given $U\in \cN(G)$, we write $\rm{FS}(U) = \{F\in \rm{FS}(G): F\subseteq U\}$.

\begin{defin}
	\label{Def:Seminorm}
	A \emph{semi-norm} on $G$ is a bounded, symmetric, continuous function $\sigma\colon G\to \bbR^{{\geq}0}$ with $\sigma(e_G) = 0$ and $\sigma(gh)\leq \sigma(g)+\sigma(h)$ for all $g, h\in G$. We set $\|\sigma\| = \sup\{\sigma(g): g\in G\}$.  If $c> 0$, we set $\rmB_\sigma(c) = \{g\in G: \sigma(g)< c\}$. We write $\rm{SN}(G)$ for the set of semi-norms on $G$. Given $s> 0$, write $\rm{SN}^s(G):= \rm{SN}(G)\cap \rmC^s(G)$. We call $\sigma\in \rm{SN}(G)$ a \emph{norm} if additionally $\sigma(g) = 0$ implies $g = e_G$. 
	
	If $\sigma, \sigma'\in \rm{SN}(G)$, we write $\sigma\leq \sigma'$ iff $\sigma(g)\leq \sigma'(g)$ for each $g\in G$. As $\max\{\sigma, \sigma'\}\in \rm{SN}(G)$, we see that for every $s> 0$, $(\rm{SN}^1(G), \leq)$ is a directed partial order.
\end{defin}

There is a 1-1 correspondence between semi-norms on $G$ and continuous right-invariant pseudo-metrics on $G$; given $\sigma\in \rm{SN}(G)$, we obtain the  pseudo-metric $\rho_\sigma$ via $\rho_\sigma(g, h) = \sigma(gh^{-1})$. Conversely, if $\rho$ is a continuous right-invariant pseudometric, we obtain the semi-norm $\sigma_\rho$ via $\sigma_\rho(g) = \rho(g, e_G)$. Similarly for continuous left-invariant pseudometrics.

We now describe a general procedure for producing semi-norms on $G$. 

\begin{defin}
	\label{Def:Generate_SN}
	Suppose $P\subseteq G\times \bbR^{{\geq}0}$ satisfies the following:
	\begin{enumerate}
		\item 
		$\forall g\in G\, \exists c\geq 0$ with  $(g, c)\in P$,
		\item
		The function $g\to \inf\{c\geq 0: (g, c)\in P\}$ is bounded and symmetric.
		\item
		$\displaystyle\lim_{g\to e_G} \inf\{c\in \bbR^{{\geq}0}: (g, c)\in P\} = 0$
	\end{enumerate}
	We define $[[ P]] \in \rm{SN}(G)$, the \emph{semi-norm generated by $P$}, via 
	$$[[ P]](g) = \inf\left\{\sum_{i< k} c_i: \exists g_0,..., g_k\in G \text{ with } (g_i, c_i)\in P \text{ and } g = g_0\cdots g_k\right\}.$$ 
	It is straightforward to verify that $[[ P]]$ is in fact a semi-norm. \qed
\end{defin}

One instance of Definition~\ref{Def:Generate_SN} that we will use frequently is the following. 

\begin{notation}
	\label{Notation:SN_from_sequence}
	Suppose $\vec{U} = \la U_n: n< \omega\ra$ is a sequence from $\cN(G)$ with $U_{n+1}\subseteq U_n$ for each $n< \omega$ and that $\vec{c} = \la c_n: n< \omega\ra$ is a sequence from $\bbR^{{\geq}0}$ satisfying $1\geq c_0\geq c_1\cdots$ and $\lim c_n = 0$. We set $\sigma_{\vec{U}, \vec{c}} = [[ (G\times \{1\})\cup \bigcup_{n< \omega} (U_n\times \{c_n\})]]\in \rm{SN}^1(G)$.  In the case that $c_n = 2^{-n}$ for each $n< \omega$, we omit it from the notation. 
\end{notation}

Using this notation, we note the following key lemma in the proof of the Birkhoff-Kakutani metrization theorem.

\begin{fact}[\cite{Berberian}, p.\ 28]
	\label{Fact:Birkhoff_Kakutani}
	If $\vec{U} = \la U_n: n< \omega\ra$ is a sequence from $\cN(G)$ with $U_{n+1}^3\subseteq U_n$ for each $n< \omega$ and $h\in U_n\setminus U_{n+1}$, we have $2^{-n-1}\leq \sigma_{\vec{U}}(h)\leq 2^{-n}$.
\end{fact}

We remark that for any $\sigma\in \rm{SN}^1(G)$, there is $\vec{U}$ as in Fact~\ref{Fact:Birkhoff_Kakutani} with $\sigma\leq \sigma_{\vec{U}}$.

A \emph{left $G$-space} is a topological space $X$ equipped with a continuous left action $a\colon G\times X\to X$. Usually $a$ is omitted from the notation, and one writes $gx$ for $a(g, x)$. A \emph{$G$-flow} is a compact $G$-space. Given $G$-spaces $X$ and $Y$, a function $\phi\colon X\to Y$ is a \emph{$G$-map} if it is continuous and $G$-equivariant. A \emph{factor map} from $X$ onto $Y$ is a surjective $G$-map, and $Y$ is a \emph{factor} of $X$ if there is a factor map $\phi\colon X\to Y$.  One can also consider \emph{right} $G$-spaces and $G$-flows in the obvious manner. 

If $X$ is a $G$-space, then $G$ acts on $\rmC(X)$ on the right where given $p\in \rmC(X)$, $g\in G$, and $x\in X$, we set $(pg)(x) = p(gx)$. However, the action may not be continuous.

\begin{defin}
	\label{Def:G_Cont}
	If $X$ is a $G_{dsc}$-space, a function $p\in \rmC(X)$ is \emph{$G$-continuous} if the map $\lambda_p\colon G\to \rmC(X)$ given by $\lambda_p(g) = pg$ is norm continuous. Write $\rmC_G(X)\subseteq \rmC(X)$ for the subalgebra of $G$-continuous functions, and note that $\rmC_G(X)$ with the norm topology is a right $G$-space. If $s> 0$, write $\rmC_G^s(X):= \rmC_G(X)\cap \rmC^s(X)$.
\end{defin}

We emphasize in particular that $\rmC_G(X)$ is norm closed.

\begin{exa}
	\label{Exa:RUC}
	If $X = G$ is viewed as a left $G$-space in the typical way, we have $\rmC_G(G) = \rm{RUC}(G)$, the algebra of bounded, right-uniformly continuous functions on $G$. Recall that $p\colon G\to \bbC$ is right-uniformly continuous iff for any $\epsilon > 0$, there is $U\in \cN(G)$ such that $gh^{-1}\in U$ implies $|p(g) - p(h)|< \epsilon$.
\end{exa}

Sometimes, we will want a more quantitative way of describing $G$-continuity.

\begin{defin}
	\label{Def:OL}
	Let $X$ be a $G_{dsc}$-space, and fix $\sigma\colon G\to \bbR^{{\geq}0}$. We say that $p\in \rmC(X)$ is \emph{$\sigma$-orbit-Lipschitz} if whenever $g\in G$, we have $\|pg-p\|\leq \sigma(g)$. Write $\rmC_\sigma(X)$ for the $\sigma$-orbit-Lipschitz members of $\rmC(X)$. If $s> 0$, write $\rmC_\sigma^s(X):= \rmC_\sigma(X)\cap \rmC^s(X)$.
\end{defin}

When $\sigma \in \rm{SN}(G)$, we clearly have $\rmC_\sigma(X)\subseteq \rmC_G(X)$. Conversely, every $p\in \rmC_G(X)$ is $\sigma$-orbit-Lipschitz for some $\sigma\in \sn(G)$; indeed, consider $\sigma(g) = \|pg-p\|$. We also note that viewing $G$ as a $G$-space as in Example~\ref{Exa:RUC}, we have for any $\sigma\in \sn(G)$ that $\sigma\in \rmC_\sigma(G)$; indeed, for any $g, h\in G$, we have $\sigma(gh)-\sigma(h)\leq (\sigma(g)+\sigma(h))-\sigma(h) = \sigma(g)$ and $\sigma(h) - \sigma(gh)\leq (\sigma(g^{-1}) + \sigma(gh)) - \sigma(gh) = \sigma(g^{-1}) = \sigma(g)$.

If $X$ is a $G$-flow, it is straightforward to check that $\rmC_G(X) = \rmC(X)$. When $X$ is a non-compact $G$-space, we can use $\rmC_G(X)$ to create a useful compactification of $X$.

\begin{notation}
	\label{Notation:Alpha_G}
	If $X$ is a $G_{dsc}$-space, we write $\alpha_G(X) = \wh{\rmC_G(X)}$. The map $\iota_X^G\colon X\to \alpha_G(X)$ is defined so that for any $x\in X$ and $f\in \rmC_G(X)$, we have  $\iota_X^G(x)(f) = f(x)$.
\end{notation}

\begin{defin}
	\label{Def:G_compactification}
	Given a $G_{dsc}$-space $X$, a \emph{$G$-compactification} of $X$ is a pair $(Y, \iota)$, where $Y$ is a $G$-flow and $\iota\colon X\to Y$ is a $G$-map with dense image.  
\end{defin}

The next fact, while only stated for $G$-spaces in \cite{de_Vries_1977}, easily extends to $G_{dsc}$-spaces.

\begin{fact}[\cite{de_Vries_1977}]	
	\label{Fact:G_Compactification}
	Given a $G_{dsc}$-space $X$, then $(\alpha_G(X), \iota_X^G)$ is the \emph{maximal $G$-equivariant compactification of $X$}, i.e.\ if $(Y, \phi)$ is any other $G$-compactification of $X$, then there is a $G$-map $\wt{\phi}\colon \alpha_G(X)\to Y$ with $\phi = \wt{\phi}\circ \iota_X^G$.
\end{fact}

In full generality, understanding $\alpha_G(X)$ given $X$ can be a difficult problem. For example, Pestov \cite{Pestov_2017} has exhibited a Polish group $G$ and a faithful $G$-space $X$ with $\alpha_G(X)$ a singleton. Luckily, when we consider non-compact $G$-spaces, these will mostly be of the form $X = \bigsqcup_{i\in I} X_i$ where each $X_i$ is a $G$-flow. It is routine to see that in this case, $G$-continuous functions separate points from closed sets not containing them, hence $\iota_X^G$ is an embedding. When this holds, we suppress the notation $\iota_X^G$ and view $X$ as a subspace of $\alpha_G(X)$.

\begin{exa}
	\label{Exa:Samuel}
	With $X = G$ as in Example~\ref{Exa:RUC}, we have $\alpha_G(G) = \rm{Sa}(G)$, the \emph{Samuel compactification} of $G$. In a mild abuse of notation, we often identify $\rmC(\sa(G))$ and $\rmC_G(G)= \rm{RUC}(G)$. 
\end{exa}

\section{Ultracoproducts of $G$-flows}
\label{Section:Ultracoproduct_Flows}

For the time being, we fix the following notation. Let $I$ be an infinite set and $\la X_i: i\in I\ra$ a tuple of $G$-flows. Form the $G$-space $X:= \bigsqcup \vec{X}$ and the compactification $\alpha_G(X)$. We can view $\beta I$ as a \emph{motionless} $G$-flow, i.e.\ where $g\cU = \cU$ for every $g\in G$ and $\cU\in \beta I$. The map $\pi_I\colon X\to \beta I$ given by $\pi_I(x) = i$ iff $x\in X_i$ is a $G$-map, hence it continuously extends to $\alpha_G(X)$, and we also denote this continuous extension by $\pi_I$. 

\begin{defin}
	\label{Def:Ultracoproduct}
	Let $\cU\in \beta I$. The \emph{$G$-equicontinuous ultracoproduct of $\la X_i: i\in I\ra$ along $\cU$}, denoted $\Sigma_\cU^G X_i$, denotes the $G$-flow $\pi_I^{-1}(\{\cU\})\subseteq \alpha_G(X)$. When $G$ is discrete, we omit it from the notation. In the case $X_i = Y$ and $X = I\times Y$, we call $\Sigma_\cU^G Y$ the \emph{$G$-equicontinuous ultracopower} of $Y$ along $\cU$. 
	
	If $\la Y_i: i\in I\ra$ are $G$-flows and $\phi_i\colon X_i\to Y_i$ are $G$-maps, the map $\Sigma_\cU^G \phi_i\colon \Sigma_\cU^GX_i\to \Sigma_\cU^GY_i$ is the restriction to $\Sigma_\cU^GX_i$ of the continuous extension of $\bigsqcup_{i\in I} \phi_i\colon \bigsqcup X_i\to \alpha_G(\bigsqcup_{i\in I}Y_i)$. 
\end{defin}

We can dualize to obtain a notion of $G$-equicontinuous ultraproduct for the corresponding $C^*$-algebras. We can identify $\rmC_G(X)$ with the set
$$\bigcup_{s>0} \bigcup_{\sigma\in \rm{SN}(G)} \prod_{i\in I} \rmC_\sigma^s(X_i).$$
Under this identification, we then have $\rmC(\Sigma_\cU^G X_i)\cong \rmC_G(X)/\sim_\cU$, where $\sim_\cU$ is exactly as defined in Section~\ref{Section:Ultracoproducts_Spaces}. Given $(p_i)_{i\in I}\in \rmC_G(X)$, we write $(p_i)_\cU^G\in \rmC(\Sigma_\cU^G X_i)$ for the corresponding continuous function on the $G$-equicontinuous ultracoproduct. Conversely, note that if $p\in \rmC(\Sigma_\cU^G X_i)$, then the set of continuous extensions of $p$ to some $\tilde{p}\in \rmC(\alpha_G(X))$ are in canonical one-one correspondence with those $(p_i)_{i\in I}\in \rmC_G(X)$ with $(p_i)_\cU^G = p$.

The terminology is borrowed from \cite{BYG_2021}, where the authors consider two methods of forming an ultraproduct for unitary representations of locally compact groups. One corresponds to our Definition~\ref{Def:Ultracoproduct}. The other corresponds to instead forming $\alpha_G(\Sigma_\cU X_i)$; call this the \emph{$G$-continuous ultraproduct}. We identify continuous functions on $\alpha_G(\Sigma_\cU X_i)$ with the set
\begin{align*}
\{(p_i)_{i\in I}\in \bigcup_{s>0} \rmC^s(X_i):\,\,  &\exists \sigma\in \rm{SN}(G)\, \forall F\subseteq_\fin G\, \forall \epsilon > 0\\
&\{i\in I: \forall g\in F\, \|p_ig - p_i\|\leq \sigma(g)+\epsilon \}\in \cU  \}/\sim_\cU.
\end{align*}
As $\rmC(\Sigma_\cU^G X_i)\subseteq \rmC_G(\Sigma_\cU X_i)$, we obtain a factor map from the $G$-continuous ultracoproduct to the $G$-equicontinuous one. An earlier version of this paper claimed that for locally compact $G$, the $G$-continuous and $G$-equicontinuous ultracoproducts coincided. However, this is not true. 

\begin{prop}
	\label{Prop:Different_Ultracoproducts}
	If $G$ is a non-discrete topological group, then there are an infinite set $I$ and $\cU\in \beta I$ with $\rmC(\Sigma_\cU^G \sa(G))\subsetneq \rmC_G(\Sigma_\cU \sa(G))$
\end{prop}

\begin{proof}
	Let $I = \rm{SN}^1(G)\times \rm{FS}(G) \times \bbR^{{>}0}$. As in Example~\ref{Exa:Samuel}, we identify continuous functions on $\sa(G)$ with members of $\rmC_G(G)$. With this identificaton in mind, given $i = (\sigma, F, \epsilon)\in I$, we set 
	$$p_i = [[ \sigma\cup \{(g, \epsilon): g\in F\}]]\in \sn^1(G).$$
	\begin{claim}
		For any $F\in \rm{FS}(G)$, any $\epsilon > 0$, and any $U\in \cN(G)$, there is $\sigma\in \rm{SN}^1(G)$ such that, writing $i = (\sigma, F, \epsilon)$, we have $U\not\subseteq \rmB_{p_i}(1)$.
	\end{claim}

	\begin{proof}[Proof of claim]
		Let $n< \omega$ satisfy $\epsilon n\geq 1$. Using continuity of the action and non-discreteness of the group, find $V\in \cN(G)$ such that $F\cap V = \{e_G\}$ and $U\not\subseteq (VF)^nV$. Find $\sigma\in \rm{SN}^1(G)$ with $\rmB_\sigma(1)\subseteq V$, and set $i = (\sigma, F, \epsilon)$. Observe that $\rmB_{p_i}(1)\subseteq (VF)^nV$, hence $U\not\subseteq \rmB_{p_i}(1)$. 
	\end{proof}
	Let $\cU\in \beta I$ be any ultrafilter satisfying both of the following:
	\begin{itemize}
		\item 
		Given $F\in \rm{FS}(G)$ and $\epsilon > 0$, $\{(\sigma, F', \epsilon')\in I: \sigma\in \sn^1(G), F'\supseteq F\text{ and } \epsilon'< \epsilon\}\in \cU$.
		\item
		For each $U\in \cN(G)$, the set $\{i\in I: U\not\subseteq \rmB_{p_i}(1)\}\in \cU$.
	\end{itemize} 
	By the first item, $(p_i)_\cU$ is constant, so certainly $G$-continuous. By the second, there is no $(q_i)_{i\in I}\in \rmC_G(I\times\sa(G))$ with $(p_i)_{i\in I}\sim_\cU (q_i)_{i\in I}$.  
\end{proof}

As the equicontinuous version will be the main notion of ultracoproduct that we consider, we shorten its name to \emph{$G$-ultracoproduct}. To us, this seems to be the ``correct" choice of $G$-ultracoproduct. It is the version that stays entirely within the realm of $G$-flows for the \emph{topological group} $G$ and that yields a notion of weak containment that is dynamically meaningful (see Theorem~\ref{Thm:Weak_Cont_Vietoris}). One downside though (or upside, depending on your point of view) is that this choice makes the theory of weak containment much more subtle.

Our earlier discussion on ultraproducts and ultracoproducts of compact spaces suggests how to form the $G$-ultra\emph{product} of the family $\la X_i: i\in I\ra$. Namely, we define $\lim_{\vec{X},\cU}^G\colon \prod_{i\in I} X_i\to \Sigma_\cU^G X_i$ and $\iota_{\vec{X},\cU}^G\colon \Pi_\cU X_i\to \Sigma_\cU^G X_i$  almost exactly as before, but this time, we \emph{define} the $G$-ultraproduct $\Pi_\cU^G X$ to be (homeomorphic to) the image of $\iota_{\vec{X},\cU}^G$. However, when $G$ is non-discrete, the map $\iota_{\vec{X},\cU}^G$ need not be injective.

In the case that $X_i = X$ for a fixed $G$-flow $X$, we call $\Sigma_\cU^G X$ the \emph{ultracopower} of $X$ along $\cU$, and we can form the ultracopower $G$-map $\pi_{X, \cU}^G\colon \Sigma_\cU^G X\to X$  by continuously extending the projection $I\times X\to X$ to $\alpha_G(I\times X)$, then restricting to $\Sigma_\cU^G X$. We also have the (set-theoretic) ultrapower embedding $\delta_{X,\cU}\colon X\to \Pi_\cU X$ as before, and therefore a map $j_{X,\cU}^G:= \iota_{X,\cU}^G\circ \delta_{X,\cU}\colon X\to \Sigma_\cU^G X$. While $j_{X,\cU}^G$ is always injective (indeed $\pi_X\circ j_{X,\cU}^G = \rm{id}_X$) and $G$-equivariant, it is in general not continuous.

\subsection{The Vietoris topology and weak containment}
Recall that if $Z$ is a compact Hausdorff space, then we equip $\exp(Z)$ with the \emph{Vietoris topology}, which is also compact Hausdorff. The typical basic open set in $\exp(Z)$ has the form
$$\rmN_{Q}:= \{K\in \exp(Z): K\subseteq \bigcup Q\text{ and }\forall\, A\in Q\, (K\cap A\neq \emptyset)\}$$
where $Q\in [\op(Z)]^{{<}\omega}$. If $\cB$ is a basis for the topology on $Z$, we may restrict our attention to the case $Q\subseteq \cB$. Another way of describing the Vietoris topology is that given a net $(Y_j)_{j\in J_0}$ from $\exp(Z)$ and $Y\in \exp(Z)$, we have $Y_j\to Y$ if both:
\begin{itemize}
	\item
	For any subnet $(Y_\alpha)_{\alpha\in J_1}$ and any $y_\alpha\in Y_\alpha$ with $(y_\alpha)_{\alpha\in J_1}$ convergent, we have $\lim y_\alpha\in Y$.
	\item 
	For any $y\in Y$, there is a subnet $(Y_\alpha)_{\alpha \in J_1}$ and $y_\alpha\in Y_\alpha$ with $\lim y_\alpha = y$. 
\end{itemize}
We remark that in both items, we must allow the possibility of passing to a subnet.

It turns out that the interaction between ultracoproducts of $G$-flows and Vietoris limits of subflows is quite fruitful. If $Z$ is a $G$-flow, we let $\rm{Sub}_G(Z)\subseteq \exp(Z)$ denote the closed subspace of $\exp(Z)$ consisting of $G$-subflows. 

\begin{lemma}
	\label{Lem:Vietoris}
	Given $\vec{X} = \la X_i: i\in I\ra$ a tuple of $G$-flows and writing $X = \alpha_G(\bigsqcup_{i\in I}X_i)$, we have $\Sigma_\cU^G X_i = \displaystyle\lim_{i\to \cU} X_i$ in $\rm{Sub}_G(\alpha_G(X))$.
\end{lemma}

\begin{proof}
	As $\subg(\alpha_G(X))$ is compact, we may assume that the limit on the right hand side exists; call it $Y$. Clearly $Y\subseteq \Sigma_\cU^G X_i$. For the other inclusion, let $z\in \Sigma_\cU^G X_i$, and consider some $A \in \op(z, \alpha_G(X))$. As $\alpha_G(X)$ is a compactification of $X$, we have $A\cap X\neq \emptyset$, and from here it is routine to construct the needed subnet $(X_\alpha)_{\alpha\in J}$ and points $x_\alpha\in X_\alpha$ with $\lim x_\alpha = z$.
\end{proof}

\begin{prop}
	\label{Prop:Vietoris_Limit_WC}
	For any $G$-flow $Z$, $Y\in \rm{Sub}_G(Z)$, and net $(Y_i)_{i\in I}$ from $\rm{Sub}_G(Z)$ with $Y_i\to Y$, there is some ultracoproduct of the $Y_i$ which factors onto $Y$.  
\end{prop}	

\begin{proof}
	Letting $\leq_I$ denote the upwards-directed partial order on $I$, let $\cU\in \beta I$ be an ultrafilter such that for every $i\in I$, we have $\{j\in I: j\geq_I i\}\in \cU$. Let $\phi_i\colon Y_i\to Y_i$ be the identity for each $i\in I$, and let $\phi\colon \alpha_G(\bigsqcup_{i\in I} Y_i)\to Z$ be the continuous extension of the union of the $\phi_i$. As $\phi$ induces a continuous map between the respective Vietoris spaces, Lemma~\ref{Lem:Vietoris} yields $\phi[\Sigma_\cU^G Y_i] = Y$ as desired. 
\end{proof}

The conclusion of Proposition~\ref{Prop:Vietoris_Limit_WC} suggests the following relation between two $G$-flows.

\begin{defin}
	\label{Def:Weak_Containment}
	Let $X$ and $Y$ be $G$-flows. We say that $X$ is \emph{weakly contained} in $Y$ and write $X\preceq_G Y$ if $X$ is a factor of some ultracopower of $Y$. We say that $X$ and $Y$ are \emph{weakly equivalent} and write $X\sim_G Y$ if both $X\preceq_G Y$ and $Y\preceq_G X$.
\end{defin}

One of the main goals of this paper is to find sufficient conditions which ensure that weak containment is a pre-order and that weak equivalence is an equivalence relation. For now, we end the section with the following alternative characterization of weak containment.

\begin{theorem}
	\label{Thm:Weak_Cont_Vietoris}
	Given $G$-flows $X$ and $Y$, the following are equivalent.
	\begin{enumerate}
		\item 
		$X\preceq_GY$
		\item
		There are a $G$-flow $Z$ and a net $(Y_i)_{i\in I}$ from $\rm{Sub}_G(Z)$ with $Y_i\cong Y$ for each $i\in I$ and with $Y_i\to X'\in \rm{Sub}_G(Z)$ for some $X'\cong X$. 
	\end{enumerate}
\end{theorem}

\begin{proof}
	$(2)\Rightarrow (1)$ follows from Proposition~\ref{Prop:Vietoris_Limit_WC}. For $(1)\Rightarrow (2)$, suppose $X\preceq_GY$ as witnessed by the set $I$, $\cU\in \beta I$, and factor map $\phi\colon \Sigma_\cU^GY\to X$. Let $E_\phi\subseteq (\Sigma_\cU^GY)^2$ denote the associated equivalence relation. Viewing $E_\phi\subseteq \alpha_G(I\times Y)^2$, $E_\phi$ is still a closed, $G$-invariant equivalence relation, so let $Z = \alpha_G(I\times Y)/E_\phi$, and write $\pi_\phi\colon \alpha_G(I\times Y)\to Z$ for the quotient map. Given $i\in I$, let $Y_i = \pi_\phi[\{i\}\times Y]$, and let $X' = \pi_\phi[\Sigma_\cU^GY]$. Then $Y_i\cong Y$, $X'\cong X$, and $\lim_{i\to \cU} Y_i = X'$.
\end{proof}

\section{Fubini sums and Tietze extensions}
\label{Section:Fubini}

Given infinite sets $I, J$ and ultrafilters $\cU\in \beta I$ and $\cV\in \beta J$, the \emph{Fubini sum} of $\cU$ and $\cV$, sometimes called the \emph{tensor product}, is the ultrafilter $\cU\otimes \cV\in \beta(I\times J)\setminus (I\times J)$ where given $A\subseteq I\times J$, we have
\begin{align*}
A\in \cU\otimes \cV &\Leftrightarrow \forall^\cU i\in I\, \forall^\cV j\in J\, (i, j)\in A\\ &\Leftrightarrow \{i\in I: \{j\in J: (i, j)\in A\}\in \cV\}\in \cU.
\end{align*}

Fubini sums of ultrafilters show up upon considering what happens upon taking the ultracopower of an ultracopower. Suppose $X$ is a $G$-flow, and consider first forming $\Sigma_\cV^G X\subseteq \alpha_G(J\times X)$, then forming $\Sigma_\cU^G \Sigma_\cV^G X\subseteq \alpha_G(I\times (\Sigma_\cV^G X))$. Compare this to $\Sigma_{\cU\otimes\cV}^G X\subseteq \alpha_G(I\times J\times X)$. Note that $\alpha_G(\{i\}\times J\times X)\subseteq \alpha_G(I\times J\times X)$, so let $\psi\colon \alpha_G(I\times (\Sigma_\cV^G X))\to \alpha_G(I\times J\times Z)$ be the map which sends $\{i\}\times \Sigma_\cV^G X$ to its natural copy inside $\alpha_G(I\times J\times X)$, then continuously extend. In particular, note that $\psi[\Sigma_\cU^G\Sigma_\cV^G X] = \Sigma_{\cU\otimes \cV}^G X$. We call $\psi$ the \emph{canonical factor map} from $\Sigma_\cU^G\Sigma_\cV^G X$ onto $\Sigma_{\cU\otimes \cV}^G X$ (of course, $\psi$ depends on $\cU$, $\cV$, $G$, and $X$, but these will typically be clear from context). It is helpful to think about $\psi$ in terms of the dual inclusion $\hat{\psi}$ of $C^*$-algebras. Continuous functions on $\Sigma_{\cU\otimes\cV}^G X$ are represented by $G$-continuous functions on $I\times J\times X$, i.e.\ uniformly bounded tuples $(p_{ij})_{(i, j)\in I\times J}$ such that for some $\sigma\in \rm{SN}(G)$, we have $p_{ij}\in \rmC_\sigma(X)$ for every $(i, j)\in I\times J$. On the other hand, continuous functions on $\Sigma_\cU^G\Sigma_\cV^G X$ are represented by uniformly bounded tuples $(p_{ij})_{(i, j)\in I\times J}$ such that the following both hold.
\begin{itemize}
	\item 
	For each $i\in I$, $(p_{ij})_{j\in J}\in \rmC_G(J\times X)$, i.e.\ there is $\sigma_i\in \rm{SN}(G)$ with $p_{ij}\in \rmC_{\sigma_i}(G)$ for every $j\in J$. 
	\item
	$((p_{ij})_{\cV}^G)_{i\in I}\in \rmC_G(I\times \Sigma_\cV^GX)$.
\end{itemize}
It is the non-uniformity of the $\sigma_i$ in the first item which can make $\hat{\psi}$ a strict inclusion. One of the main goals of this paper is to analyze the class of flows for which this does not happen.

\begin{defin}
	\label{Def:Fubini_GFlow}
	We say that a $G$-flow $X$ is \emph{Fubini} if for any infinite sets $I, J$ and ultrafilters $\cU\in \beta I$ and $\cV\in \beta J$, the canonical factor map $\psi\colon \Sigma_\cU^G \Sigma_\cV^G  X\to\Sigma_{\cU\otimes \cV}^G X$ is an isomorphism. We say that $G$ is \emph{Fubini} if $\rm{Sa}(G)$ is a Fubini $G$-flow. 
	
	We say that a $G$-flow $X$ is \emph{weakly Fubini} if for any infinite sets $I, J$ and ultrafilters $\cU\in \beta I$, $\cV\in \beta J$, we have  $\Sigma_\cV^G \Sigma_\cU^G X\preceq_G X$, and $G$ is \emph{weakly Fubini} iff $\sa(G)$ is. Note that Fubini implies weakly Fubini.
\end{defin}

	We note that both the class of Fubini $G$-flows and the class of weakly Fubini $G$-flows are closed under ultracopowers. Whether a given ultraco\emph{product} of (weakly) Fubini flows is (weakly) Fubini seems to be much more subtle; we will provide affirmative answers in some specific cases (see Corollary~\ref{Cor:SNRPC_Ults} and Proposition~\ref{Prop:Urysohn_Ults}). Eventually, we will see that when $G$ is locally compact, then \emph{every} $G$-flow is Fubini (see Proposition~\ref{Prop:Int_k_bounded_Fubini}, Proposition~\ref{Prop:LC_Respecting}, and Theorem~\ref{Thm:Seminorm_Respecting_Fubini}).

\begin{prop}
	\label{Prop:Weak_Equiv_ER_Fubini}
	On the class of weakly Fubini $G$-flows, weak containment is a pre-order, and weak equivalence is an equivalence relation.
\end{prop}

\begin{proof}
	It suffices to prove the first claim, so let $X, Y, Z$ be weakly Fubini $G$-flows with $X\preceq_G Y$ and $Y\preceq_G Z$. There are an infinite set $I$, an ultrafilter $\cU\in \beta I$, and a factor map $\phi\colon \Sigma_\cU^G Y\to X$. Similarly,  we can find $J$, $\cV\in \beta J$, and $\xi\colon \Sigma_\cV^G Z\to Y$. Then $\phi \circ \Sigma_\cU^G \xi\colon \Sigma_\cU^G\Sigma_\cV^G Z\to X$ is a surjective $G$-map. As $Z$ is weakly Fubini, there is some infinite set $K$ and $\cW\in \beta K$ and some factor map $\theta\colon \Sigma_\cW^G Z\to \Sigma_\cU^G\Sigma_\cV^G Z$. Then $\phi\circ \Sigma_\cU^G\xi\circ \theta \colon \Sigma_\cW^G Z\to X$ witnesses that $X\preceq_G Z$.
\end{proof}

In practice, the only way developed in this paper to show that a given flow or group is weakly Fubini is to show that it is weakly Fubini. We will eventually see that when $G$ is Fubini, a wide class of $G$-flows is Fubini, including all $G$-flows when $G$ is locally compact. First, we show how the Fubini property is related to a $G$-continuous version of the Tietze extension theorem. Recall that by the Tietze extension theorem, whenever $Y\subseteq X$ are compact spaces and $f\in \rmC(Y)$, there is $\tilde{f}\in \rmC(X)$ with $\tilde{f}|_Y = f$.

\begin{defin}
	\label{Def:G_Tietze}
	Let $Y\subseteq X$ be $G$-flows. Given $\sigma_0, \sigma_1\in \rm{SN}(G)$, we say that the inclusion $Y\subseteq X$ is \emph{$(\sigma_0, \sigma_1)$-Tietze} if whenever $f\in \rmC_{\sigma_0}^1(Y)$, there is $\tilde{f}\in \rmC_{\sigma_1}(X)$ with $\tilde{f}|_Y = f$; equivalently, one can demand $\tilde{f}\in \rmC_{\sigma_1}^1(X)$.  We say that $Y\subseteq X$ is \emph{Tietze} if for any $\sigma_0\in \rm{SN}(G)$, there is $\sigma_1\in \rm{SN}(G)$ such that $Y\subseteq X$ is $(\sigma_0, \sigma_1)$-Tietze. We say that $Y\subseteq X$ is \emph{weakly Tietze} if for any $\sigma_0\in \sn(G)$ and $\delta > 0$, there is $\sigma_1\in \sn(G)$ so that for any $f\in \rmC_{\sigma_0}^1(Y)$, there is $\tilde{f}\in \rmC_{\sigma_1}(X)$ with $\|\tilde{f}|_Y-f\| < \delta$; equivalently, one can demand $\tilde{f}\in \rmC_{\sigma_1}^{1+\delta}(X)$.
\end{defin}

\begin{rem}
	$Y\subseteq X$ is $(\sigma_0, \sigma_1)$-Tietze iff it is $(\min\{\sigma_0, 1\}, \min\{\sigma_1, 1\})$-Tietze. Hence $Y\subseteq X$ is Tietze iff for every $\sigma_0\in \sn^1(G)$, there is $\sigma_1\in \sn^1(G)$ such that $Y\subseteq X$ is $(\sigma_0, \sigma_1)$-Tietze.
\end{rem}

\begin{notation}
	\label{Notation:Fubini_SN}
	Given $\sigma, \sigma'\in \rm{SN}(G)$, $F\in \rm{FS}(G)$, and $\epsilon > 0$, we set 
	$$\Phi(\sigma, \sigma', F, \epsilon) := [[ \sigma\cup \{(g, \sigma'(g)+\epsilon): g\in F\}]]\in \rm{SN}(G).$$
\end{notation}

\begin{prop}
	\label{Prop:Fubini_Flow_Tietze}
	For any $G$-flow $X$, the following are equivalent.
	\begin{enumerate}
		\item 
		$X$ is Fubini.
		\item
		For any infinite set $J$ and $\cV\in \beta J$, $\Sigma_\cV^GX\subseteq \alpha_G(J\times X)$ is weakly Tietze.
		\item
		For any $\sigma_0\in \sn(G)$ and any $\delta > 0$, there is $\sigma_1\in \sn(G)$ such that for any $\sigma_2\in\sn(G)$, there are $F\in \rm{FS}(G)$ and $\epsilon > 0$ such that for any $f\in \rmC_{\Phi(\sigma_2, \sigma_0, F, \epsilon)}^1(X)$, there is $\tilde{f}\in \rmC_{\sigma_1}(X)$ with $\|\tilde{f}-f\|\leq \delta$.
	\end{enumerate}
\end{prop}

\begin{proof}
	$(3)\Rightarrow (1)$: Towards showing $(1)$, fix infinite sets $I$ and $J$ and ultrafilters $\cU\in \beta I$ and $\cV\in \beta J$. Let $p\in \rmC^1(\fubini{\cU}{\cV}{G}X)$, towards showing that $p\in \im(\hat{\psi})$. Let $(p_i)_{i\in I}\in \rmC_G^1(I\times \Sigma_\cV^G X)$ satisfy $(p_i)_\cU^G = p$. Find $\sigma_0\in \rm{SN}(G)$ with $(p_i)_{i\in I}\in \rmC_{\sigma_0}^1(I\times \Sigma_\cV^GX)$. Fix $\delta > 0$. Find $\sigma_1\in \sn(G)$ as promised by item $(3)$. For each $i\in I$,  find $\sigma_i\in \sn(G)$ and $(p_{ij})_{j\in J}\in \rmC_{\sigma_i}^1(J\times X)$ satisfying $(p_{ij})_\cV^G = p_i$. Then given $i\in I$ and considering $\sigma_2 = \sigma_i$, find $F_i\in \rm{FS}(G)$ and $\epsilon_i> 0$ as given by $(3)$. Write $\Phi_i = \Phi(\sigma_i, \sigma_0, F_i, \epsilon_i)$. As $p_i\in \rmC_{\sigma_0}^1(\Sigma_\cV^GX)$, we must have 
	\begin{align*}
		J_i:= \{j\in J: p_{ij}\in \rmC_{\Phi_i}^1(X)\}\in \cV.
	\end{align*}
	For each $j\in J_i$, let $q_{ij}\in \rmC_{\sigma_1}(X)$ satisfy $\|q_{ij}-p_{ij}\|\leq \delta$. If $j\not\in J_i$, set $q_{ij}\equiv 0$. Then $(q_{ij})_{(i, j)\in I\times J}\in \rmC_{\sigma_1}(I\times J\times X)$, implying $((q_{ij})_\cV^G)_\cU^G\in \im(\hat{\psi})$, and also $\|((q_{ij})_\cV^G)_\cU^G - p\|\leq \delta$. As $\delta > 0$ was arbitrary and $\im(\hat{\psi})$ is norm-closed, we have $p\in \im(\hat{\psi})$. 
	\vspace{3 mm}

	$\neg(2)\Rightarrow \neg(1)$: Fix an infinite set $J$ and $\cV\in \beta J$ witnessing the failure of $(2)$. As $\Sigma_\cV^GX\subseteq \alpha_G(J\times X)$ is not weakly Tietze, find a bad $\sigma_0\in \rm{SN}(G)$ and $\delta > 0$ witnessing this. Let $\rmF\colon \sn(G)\to \rmC_{\sigma_0}^1(\Sigma_\cV^GX)$ be such that for each $\sigma_1\in \sn(G)$, the function $\rmF(\sigma_1)$ is bad, i.e.\ whenever $(q_j)_{j\in J}\in \rmC_{\sigma_1}(J\times X)$, we have $\|(q_j)_\cV^G - \rmF(\sigma_1)\|\geq \delta$.
	
	We set $I = \sn(G)$, and let $\cU\in \beta I$ be any ultrafilter such that for every $\sigma\in I$, we have $\{\sigma'\in I: \sigma\leq \sigma'\}\in \cU$. We show that $(\rmF(i))_\cU^G\in \rmC(\fubini{\cU}{\cV}{G}X)$ is not in $\im(\hat{\psi})$. Towards a contradiction, suppose $(q_{ij})_{(i, j)\in I\times J}\in \rmC_G(I\times J\times X)$ satisfied $((q_{ij})_\cV^G)_\cU^G = (\rmF(i))_\cU^G$. This would imply $\{i\in I: \|\rmF(i)-(q_{ij})_\cV^G\|< \delta\}\in \cU$. For some $\sigma\in \sn(G)$, we have $(q_{ij})_{(i, j)\in I\times J}\in \rmC_\sigma(I\times J\times X)$.  Since $\{i\in I: \sigma\leq i\}\in \cU$, we find $i\in I$ with both $\sigma \leq i$ and $\|\rmF(i) - (q_{ij})_\cV^G\|< \delta$. However, since $(q_{ij})_{j\in J}\in \rmC_\sigma(J\times X)\subseteq \rmC_i(J\times X)$, we must have $\|\rmF(i) - (q_{ij})_\cV^G\|\geq \delta$, a contradiction. 
	\vspace{3 mm}
	
	$\neg(3)\Rightarrow \neg(2)$: Let $\sigma_0\in \sn(G)$ and $\delta > 0$ witness the failure of $(3)$. Towards showing the failure of $(2)$, set $J = \rm{FS}(G)\times \bbR^{{>}0}$, and let $\cV\in \beta J$ be any ultrafilter such that for any $(F, \epsilon)\in J$, we have $\{(F', \epsilon')\in J: F'\supseteq F \text{ and } \epsilon'\leq \epsilon\}\in \cV$. Towards showing that $\Sigma_\cV^GX\subseteq \alpha_G(J\times X)$ is not weakly Tietze as witnessed by $\sigma_0$ and $\delta$, fix some $\sigma_1\in \sn(G)$. Given this $\sigma_1$, let $\sigma_2\in \sn(G)$ witness the failure of $(3)$. Given $j = (F, \epsilon)\in J$, let $p_j\in \rmC_{\Phi(\sigma_2, \sigma_0, F, \epsilon)}^1(X)\subseteq \rmC_{\sigma_2}^1(X)$ be such that whenever $\tilde{p}_j\in \rmC_{\sigma_1}(X)$, we have $\|\tilde{p}_j - p_j\|\geq \delta$. Then $(p_j)_{j\in J}\in \rmC_G^1(J\times X)$, and by our demands on $\cV\in \beta J$, we have $(p_j)_\cV^G\in \rmC_{\sigma_0}^1(\Sigma_\cV^GX)$. However, our construction of $(p_j)_{j\in J}$ ensures that for any $(\tilde{p}_j)_{j\in J}\in \rmC_{\sigma_1}(J\times X)$, we have $\|(\tilde{p}_j)_\cV^G-(p_j)_\cV^G\|\geq \delta$.
\end{proof}

In the case $X = \sa(G)$, we can say much more. We will make use of the following general fact about real-valued Lipschitz functions on metric spaces; as I couldn't find a good reference, the proof is included.

\begin{fact}
	\label{Fact:Within_Delta_Lipschitz}
	Given a set $X$, a pseudo-metric $\rho$ on $X$, $\delta > 0$, and $f\colon X\to \bbR$ a function satisfying $|f(x)- f(y)|\leq \rho(x, y)+\delta$, then there is $f'\colon X\to \bbR$ which is $\rho$-Lipschitz and with $\|f'-f\|\leq \delta/2$. If instead $f\colon X\to \bbC$, we can find $f'\colon X\to \bbC$ with $\|f'-f\|\leq \delta/\sqrt{2}$.
\end{fact}
\begin{proof}[Proof of fact]
	By compactness of the space of $\rho$-Lipschitz functions, we may assume that $X = \{x_k: k< n\}$ is finite. Let $\{I_k = [a_k, b_k] : k< n\}$ be a minimal-under-inclusion set of closed intervals or single points such that for each $k, \ell < n$, we have $I_k\subseteq [f(x_k)-\delta/2, f(x_k)+\delta/2]$ and $\min\{|r-s|: r\in I_k, s\in I_\ell\}\leq \rho(x_k, x_\ell)$. If each $I_k = \{a_k\}$, we set $f'(x_k) = a_k$. Towards a contradiction, suppose $I_0 = [a_0, b_0]$ with $a_0< b_0$. By minimality of $\{I_k: k< n\}$, there are $k, \ell< n$ so that $b_k = a_0 - \rho(x_0, x_k)$ and $a_\ell = b_0+\rho(x_0, x_\ell)$. But now $a_\ell - b_k > \rho(x_0, x_\ell)+\rho(x_0, x_k)\geq \rho(x_k, x_\ell)$, a contradiction. The claim for $f\colon X\to \bbC$ follows by running the above argument on the real and imaginary parts of $f$.
\end{proof}

\begin{theorem}
	\label{Thm:Fubini_Groups}
	For any topological group $G$, the following are equivalent.
	\begin{enumerate}
		\item 
		$G$ is Fubini.
		\item
		For any $\sigma_0\in \rm{SN}^1(G)$, there is $\sigma_1\in \rm{SN}^1(G)$ such that for any $\sigma_2\in \rm{SN}^1(G)$ and any $\delta> 0$, there are $F\in \rm{FS}(G)$ and $\epsilon > 0$ such that pointwise, we have
		$$\Phi(\sigma_2, \sigma_0, F, \epsilon) \leq \sigma_1+\delta.$$
		\item
		For any infinite set $J$ and $\cV\in \beta J$, $\Sigma_\cV^G \sa(G)\subseteq \alpha_G(J\times \sa(G))$ is Tietze.
		\item
		For any $\sigma_0\in \rm{SN}^1(G)$ and any $\delta > 0$, there is $\sigma_1\in \rm{SN}^1(G)$ such that for every $\sigma_2\in \rm{SN}^1(G)$, there are $F\in \rm{FS}(G)$ and $\epsilon > 0$ such that pointwise, we have
		$$\Phi(\sigma_2, \sigma_0, F, \epsilon) \leq \sigma_1+\delta.$$
		\item
		For any infinite set $J$ and $\cV\in \beta J$, $\Sigma_\cV^G \sa(G)\subseteq \alpha_G(J\times \sa(G))$ is weakly Tietze.
	\end{enumerate}
\end{theorem}
	
\begin{proof}
	$(1)\Leftrightarrow (5)$ follows from Proposition~\ref{Prop:Fubini_Flow_Tietze}. $(3)\Rightarrow (5)$ and $(2)\Rightarrow (4)$ are clear.
	\vspace{3 mm}
	
	$\neg (2)\Rightarrow \neg (1)$: Suppose (2) fails, as witnessed by some bad $\sigma_0\in \rm{SN}^1(G)$ which we now fix. Then, let $\rmS\colon \rm{SN}^1(G)\to \rm{SN}^1(G)$ and $\rmD\colon \rm{SN}^1(G)\to \bbR^{{>}0}$ be such that for each $\sigma_1\in \rm{SN}^1(G)$, $\rmS(\sigma_1)$ is the bad $\sigma_2$ and $\rmD(\sigma_1)$ is the bad $\delta$. In particular, for every $\sigma_1\in \rm{SN}^1(G)$,  $F\in \rm{FS}(G)$, and $\epsilon > 0$, we have that
	$$\Phi(\rmS(\sigma_1), \sigma_0, F, \epsilon)\not\leq \sigma_1+\rmD(\sigma_1).$$ 
	
	We set $I = [\rm{SN}^1(G)]^{<\omega}$ and $J = \rm{FS}(G)\times \bbR^{{>}0}$. Given $i \in I$ and $j = (F, \epsilon)\in J$, set 
	$$p_{ij}:= \Phi(\max\{\rmS(\sigma): \sigma\in i\}, \sigma_0, F, \epsilon),$$
	where $\max\{\rmS(\sigma): \sigma\in i\}\in \rm{SN}(G)$ is the pointwise maximum.
	Let $\cU\in \beta I$ be any ultrafilter such that for every $i\in I$, we have $\{i'\in I: i'\supseteq i\}\in \cU$, and let $\cV\in \beta J$ be any ultrafilter such that for every $(F, \epsilon)\in J$, we have $\{(F', \epsilon'): F'\supseteq F\text{ and }\epsilon'\leq \epsilon\}\in \cV$. Then for each $i\in I$, we have $(p_{ij})_\cV^G \in \rmC_{\sigma_0}(\Sigma_\cV^G \sa(G))$, so in particular $((p_{ij})_\cV^G)_\cU^G\in \rmC(\Sigma_\cU^G\Sigma_\cV^G \sa(G))$. 
	
	We finish by showing that there is no $(q_{ij})_{(i,j)\in I\times J}\in \rmC_G(I\times J\times \sa(G))$ with $((p_{ij})_\cV^G)_{i\in I} \sim_\cU ((q_{ij})_\cV^G)_{i\in I}$. Towards a contradiction, suppose there was such a $(q_{ij})_{(i,j)\in I\times J}\in \rmC_{\sigma_1}(I\times J\times \sa(G))$ for some $\sigma_1\in \rm{SN}(G)$. We observe that 
	$$((p_{ij})_\cV^G)_{i\in I} \sim_\cU ((q_{ij})_\cV^G)_{i\in I}\Leftrightarrow (p_{ij})_{(i, j)\in I\times J}\sim_{\cU\otimes \cV} (q_{ij})_{(i,j)\in I\times J}.$$
	Thus for each $\delta > 0$, we can find $A_\delta\in \cU\otimes \cV$ such that $\|p_{ij}-q_{ij}\| \leq \delta$ for each $(i, j)\in A_\delta$. For each $i\in I$, let $A_\delta^i = \{j\in J: (i, j)\in A_\delta\}$; consider $\delta = \rmD(\sigma_1)$, and set $B:= \{i\in I: A_{\rmD(\sigma_1)}^i\in \cV\}\in \cU$. By our demand on $\cU$, we can find $i\in B$ with $\sigma_1\in i$, which we now fix. In particular, for every $j\in J$, we have that $p_{ij}\not\leq \sigma_1+\rmD(\sigma_1)$. But since $p_{ij}(e_G) = 0$, it follows that for $j\in A_{\rmD(\sigma_1)}^{i}$, we have $q_{ij}\not\in \rmC_{\sigma_1}(I\times J\times \sa(G))$, a contradiction.
	\vspace{3 mm}
	
	$(2)\Rightarrow (3)$: Suppose $(2)$ holds, and fix an infinite set $J$ and $\cV\in \beta J$. Given $\sigma_0\in \rm{SN}^1(G)$, let $\sigma_1\in \rm{SN}^1(G)$ be as guaranteed by $(2)$. We show that $\Sigma_\cV^G \sa(G)\subseteq \alpha_G(J\times \sa(G))$ is $(\sigma_0, \sigma_1)$-Tietze. Let $p\in \rmC_{\sigma_0}(\Sigma_\cV^J \sa(G))$, and let $(p_j)_{j\in J}\in \rmC_G(J\times G)$ satisfy $(p_j)_\cV^G = p$. For some $\sigma_2\in \rm{SN}^1(G)$, we have $(p_j)_{j\in J}\in \rmC_{\sigma_2}(J\times G)$. It follows that for every finite $F\subseteq G$ and $\epsilon > 0$, we have that 
	\begin{align*}
		&\{j\in J: \forall g\in F\, \|p_{j} - p_{j}g\|\leq \sigma_0(g)+ \epsilon\}\in \cV\\
		\Rightarrow \,\, &\{j\in J: \forall h\in G\, \|p_{j} - p_{j}h\|\leq \Phi(\sigma_2, \sigma_0, F, \epsilon)(h)\}\in \cV.
	\end{align*}  
	Given $\delta >0$, there are $F\in \rm{FS}(G)$ and $\epsilon> 0$ such that 
	$$\Phi(\sigma_2, \sigma_0, F, \epsilon) \leq \sigma_1+\delta.$$
	Hence for each $0< n< \omega$, we have
	$$J_{n}:=\{j\in J: \forall h\in G\, \|p_{j} - p_{j}h\|\leq \sigma_1(h)+1/n\}\in \cV.$$

	For each $0< n< \omega$ and $j\in J_{n}\setminus J_{n+1}$, use Fact~\ref{Fact:Within_Delta_Lipschitz} to find $q_{j}\in \rmC_{\sigma_1}(G)$ satisfying $\|p_{j} - q_{j}\|\leq 1/n$. For $j\in J\setminus J_1$, set $q_j\equiv 0$. Then $(q_{j})_{j\in J}\in \rmC_{\sigma_1}(J\times \sa(G))$ and $((p_{ij})_\cV^G)_{i\in I} \sim_\cU ((q_{ij})_\cV^G)_{i\in I}$.
	\vspace{3 mm}
	
	$(4)\Rightarrow (1)$. The proof is very similar to the proof of $(3)\Rightarrow (1)$ from Proposition~\ref{Prop:Fubini_Flow_Tietze}. Writing $X = \sa(G)$, the proof is almost identical except for finding the functions $q_{ij}$. This time, we have for each $i\in I$ and $j\in J_i$ that $p_{ij}\in \rmC_{\Phi_i}(X)\subseteq\rmC_{\sigma_1+\delta}(X)$. We then use Fact~\ref{Fact:Within_Delta_Lipschitz} to find $q_{ij}$ with $\|q_{ij}-p_{ij}\| \leq \delta$. The rest of the proof is identical.
\end{proof}

We end the section by giving some examples and non-examples of Fubini groups.

\begin{defin}
	\label{Def:Int_k_bounded}
	Given $0< k< \omega$, we say that $S\subseteq G$ is \emph{$k$-bounded} if for any $V\in \cN(G)$, there is a finite $F\subseteq G$ with $S\subseteq (VF)^kV$. We say that $S\subseteq G$ is \emph{Roelcke precompact}, or RPC, if it is $1$-bounded, and we say $G$ is RPC if it is an RPC subset of itself. We say that $S\subseteq G$ is \emph{internally $k$-bounded} if for any $V\in \cN(G)$, there is a finite $F\subseteq S$ with $S\subseteq (VF)^kV$. 
\end{defin}

 Note that for each $0< k< \omega$, the set of (internally) $k$-bounded subsets of $G$ is an ideal closed under conjugation and inverses. When $k = 1$, this ideal is also closed under left and right translations, and we have:

\begin{lemma}
	\label{Lem:LRPC_1_bounded}
	If $S\subseteq G$ is $1$-bounded, then it is internally $1$-bounded.
\end{lemma}

\begin{proof}
	Fix $V\in \cN(G)$. Find $W\in \cN(G)$ with $W^2\subseteq V$. Find a finite $F\subseteq G$ with $S\subseteq WFW$; we may assume that for each $f\in F$, we have $S\cap WfW\neq\emptyset$. For each $f\in F$, pick $f'\in S\cap WfW$, and set $F'= \{f': f\in F\}$. Then $VF'V\supseteq WFW \supseteq U$.  
\end{proof}

\begin{defin}
	\label{Def:LRPC_Group}
	We say that $G$ is \emph{locally Roelcke precompact}, or LRPC, if some $U\in \cN(G)$ is RPC. Write $\ngrpc$ for the RPC members of $\cN(G)$. We call $\sigma\in \rm{SN}(G)$ \emph{RPC} if $\rmB_{\sigma}(1)$ is RPC. Write $\snrpc(G)\subseteq \sn(G)$ for the set of RPC seminorms on $G$. When $G$ is LRPC, $\snrpc(G)$ is upwards closed, and by Fact~\ref{Fact:Birkhoff_Kakutani}, it is upwards cofinal in $\rm{SN}(G)$.
\end{defin}

Roelcke precompact groups are ubiquitous throughout mathematics. Among the Polish non-Archimedean groups, a result of Tsankov \cite{Tsankov_Oligo} shows that the RPC groups are exactly those which are inverse limits of (groups isomorphic to) oligomorphic permutation groups. By a classical result due independently to Ryll-Nardzewski, Engeler, and Svenonius (see for instance \cite{Hodges}), the oligomorphic permutation groups are exactly the automorphism groups of countable, $\omega$-categorical structures. Upon generalizing to metric structures, Ben Yaacov and Tsankov \cite{BY_T_WAP} show that the RPC Polish groups are exactly the automorphism groups of $\omega$-categorical structures. Upon weakening to LRPC groups, we get an even wider class, in particular, one that contains all locally compact groups. In this case, $S\subseteq G$ is RPC iff $S$ is precompact. Hence when $G$ is locally compact, we write $\snpc(G)$ instead of $\snrpc(G)$. We refer to \cite{Zielinski_LRPC} for more discussion on LRPC groups.

\begin{prop}
	\label{Prop:Int_k_bounded_Fubini}
	For every $0< k< \omega$, if $G$ has a base of internally $k$-bounded subsets, then $G$ is Fubini. Additionally, when $G$ is LRPC and $\sigma_0\in \snrpc(G)$, then when verifying item $(2)$ of Theorem~\ref{Thm:Fubini_Groups}, we may take $\sigma_1 = \sigma_0$. 
\end{prop}

\begin{proof}
	Fix $\sigma_0\in \rm{SN}^1(G)$, towards verifying item $(2)$ of Theorem~\ref{Thm:Fubini_Groups}. By the remark after Fact~\ref{Fact:Birkhoff_Kakutani}, we may assume that $\sigma_0 = \sigma_{\vec{U}}$ for some $\vec{U} = \la U_n: n< \omega\ra$ with $U_n\in \cN(G)$ internally $k$-bounded and with $U_{n+1}^3\subseteq U_n$ for every $n< \omega$. Fix some $0< a< \omega$ such that $2^{a-1}\geq k$, set $V_n = U_{n+a}$, and set $\sigma_1 = \sigma_{\vec{V}}$. Note that if $\sigma_1(h)< 1$, then $\sigma_1(h) = 2^a\cdot \sigma_0(h)$. 
	
	Now suppose $\sigma_2\in \rm{SN}^1(G)$ and $\delta> 0$ are given. Fix $m< \omega$ with $2^{-m}\leq \frac{\delta}{3k}$. In particular, we have $\sigma_0[U_m]\leq \frac{\delta}{3k}$. We also fix $N> m$ such that $\sigma_2[U_N]\leq \frac{\delta}{3(k+1)}$. We choose $\epsilon = \frac{\delta}{3k}$. To choose $F$, for each $i\leq m$, $U_i$ is internally $k$-bounded, so we may find $F_i\in \rm{FS}(U_i)$ with $U_i\subseteq (U_NF_i)^kU_N$. We set $F = \bigcup_{i\leq m} F_i$. 
	
	Write $\Phi = \Phi(\sigma_2, \sigma_0, F, \epsilon)$. Consider some $h\in G$. If $\sigma_1(h) = 1$, there is nothing to prove, so assume $\sigma_1(h)< 1$. If $h\in U_m$, then writing $h = v_0f_0\cdots v_{k-1}f_{k-1}v_k$ with $f_j\in F_m$ for $j< k$ and $v_j\in U_N$ for $j\leq k$, we have:
	$$\Phi(h)\leq \frac{\delta(k+1)}{3(k+1)} + \frac{\delta k}{3k} + \epsilon k \leq \delta.$$
	Suppose for some $i< m$ that $h\in U_i\setminus U_{i+1}$. As $\sigma_1(h)< 1$, Fact~\ref{Fact:Birkhoff_Kakutani} gives us $\sigma_1(h) = 2^a\cdot \sigma_0(h)\geq 2^{a-i-1}\geq k\cdot 2^{-i}$. Write $h = v_0f_0\cdots v_{k-1}f_{k-1} v_k$ with $f_j\in F_i$ for $j< k$ and $v_j\in U_N$ for $j\leq k$. Then:
	$$\Phi(h)\leq \frac{\delta(k+1)}{3(k+1)} + k\cdot 2^{-i} + \epsilon k \leq \sigma_1(h)+\delta.$$
	It follows that $G$ is Fubini.
	
	When $G$ is LRPC and $\sigma_0\in \snrpc(G)$ (note that here we are not assuming $\sigma_0 = \sigma_{\vec{U}}$), we set $\sigma_1 = \sigma_0 := \sigma$. Suppose $\sigma_2\in \sn^1(G)$ and $\delta > 0$ are given, and fix $V\in \cN(G)$ with $\max\{\sigma, \sigma_2\}[V]\leq \frac{\delta}{5}$. Set $\epsilon = \frac{\delta}{5}$, and let $F\in \rm{FS}(G)$ satisfy $\rmB_{\sigma}(1) \subseteq VFV$.  Write $\Phi = \Phi(\sigma_2, \sigma, F, \epsilon)$. Given $h\in G$, if $\sigma(h) = 1$, there is nothing to show, so suppose $\sigma(h)< 1$. Write $h = v_0 f v_1$ with $v_0, v_1\in V$ and $f\in F$, and note that $f = v_0^{-1}hv_1^{-1}$. We have:
	\begin{align*}
		\Phi(h) &\leq \Phi(v_0)+ \Phi(f) + \Phi(v_1)\\
		&\leq \sigma_2(v_0) + \sigma(f) + \epsilon + \sigma_2(v_1)\\
		&\leq \sigma_2(v_0) + \sigma(v_0^{-1}) + \sigma(h)+ \sigma(v_1^{-1}) + \epsilon + \sigma_2(v_1)\\
		&\leq \sigma(h) + \delta. \qedhere
	\end{align*}
\end{proof}

\begin{cor}
	\label{Cor:LRPC_Tietze}
	When $G$ is LRPC and $\sigma\in \snrpc(G)$, then for any infinite set $J$ and $\cV\in \beta J$, we have that $\Sigma_\cV^G \sa(G)\subseteq \alpha_G(J\times \sa(G))$ is $(\sigma, \sigma)$-Tietze.
\end{cor}

\begin{proof}
	Combine Proposition~\ref{Prop:Int_k_bounded_Fubini} with the proof of $(2)\Rightarrow (3)$ from Theorem~\ref{Thm:Fubini_Groups}.
\end{proof}

Thus not only is every LRPC group Fubini, but the class of Fubini groups is strictly larger. As an example of a non-LRPC Fubini group, consider the automorphism group of the rational Urysohn space with the topology of discrete pointwise convergence (i.e.\ viewing the rational Urysohn space as a countable first-order structure). Any stabilizer of a non-empty finite subset of the rational Urysohn space is internally $2$-bounded.

In the other direction, we argue that every Fubini group must be \emph{locally bounded}; we refer to \cite{Rosendal_2021} for the definition. 

\begin{prop}
	\label{Prop:Not_Fubini}
	Suppose for every $U\in \cN(G)$, there is $V\in \cN(G)$ such that for every $k< \omega$ and $F\in \rm{FS}(G)$, we have $U\not\subseteq (VF)^kV$. Then $G$ is not Fubini. In particular, every Fubini group is locally bounded.
\end{prop}

\begin{proof}
	Consider $\sigma_0 \equiv 0$, and fix $\sigma_1\in \rm{SN}^1(G)$. Write $U = B_{\sigma_1}(1/2)$. Find $V\in \cN(G)$ such that for every $k< \omega$ and $F\in \rm{FS}(G)$, we have $U\not\subseteq (VF)^kV$. Pick $\sigma_2\in \rm{SN}^1(G)$ such that $\rmB_{\sigma_2}(1)\subseteq V$. Fix $F\in \rm{FS}(G)$ and $\epsilon > 0$, and write $\Phi = \Phi(\sigma_2, 0, F, \epsilon)$. Then if $k> 1/\epsilon$, we have $\rmB_\Phi(1)\subseteq (VF)^kV$, so in particular, $\rmB_{\sigma_1}(1/2)\not\subseteq \rmB_\Phi(1)$. Setting $\delta = 1/2$, we see that $\Phi \not\leq \sigma_1 +\delta$. Thus $G$ is not Fubini.
	
	The ``in particular'' follows from Proposition 2.15(5) of \cite{Rosendal_2021}. 
\end{proof}

We end the section with two questions. The first is straightforward.

\begin{que}
	\label{Question:Fubini}
	Are there any topological groups which are not weakly Fubini? Are there any $G$-flows which are not weakly Fubini? If $G$ is LRPC, are there any $G$-flows which are not Fubini?
\end{que}

The second question regards the complexity of the set of Polish Fubini groups. To make sense of this, we need to fix a way of discussing the collection of Polish groups as a standard Borel space. One method of doing this is to fix a universal Polish group $\sfG$, for instance the isometry group of the Urysohn space \cite{Uspenskij_Urysohn}, and view the closed subgroups of $\sfG$ as a Borel subset of $F(\sfG)$, the standard Borel space of closed subsets of $\sfG$. Using a mild modification of Theorem~12.13 from \cite{Kechris_Classical}, one can find a sequence of Borel functions $d_n\colon F(\sfG)\to \sfG$ such that $d_n(H)\in H$ and such that $\{d_n(H): n< \omega\}$ is a dense subgroup of $H$ for every closed subgroup $H\in F(\sfG)$. Then considering Theorem~\ref{Thm:Fubini_Groups}, we see that the collection of Polish Fubini groups is $\bf{\Pi}_3^1$. Can this be improved? 

\begin{que}
	\label{Question:Fubini_Borel}
	In a suitable standard Borel space of Polish groups, what is the complexity of the set of Fubini groups? Is this subspace Borel? Is it ${\bf{\Pi}}^1_3$-complete?
\end{que}

\section{Gleason complete flows and their relatives}

By definition, $G$ is Fubini iff the $G$-flow $\sa(G)$ is Fubini. This section will show that when $G$ is Fubini, then all \emph{Gleason complete} (formerly called \emph{MHP}) flows are Fubini. We will also discuss weakenings of the Gleason complete property which are sufficient for this. 

First, we discuss why Theorem~\ref{Thm:Fubini_Groups} works for $\sa(G)$, but not necessarily other $G$-flows. So suppose $G$ is Fubini and $X$ is a $G$-flow. Fix $\sigma_0\in \sn^1(G)$, and let $\sigma_1\in \sn^1(G)$ be as given by item $(2)$ of Theorem~\ref{Thm:Fubini_Groups}. Fix an infinite set $J$ and $\cV\in \beta J$. Towards attempting to show that $\Sigma_\cV^GX\subseteq \alpha_G(J\times X)$ is $(\sigma_0, \sigma_1)$-Tietze, suppose $p\in \rmC(\Sigma_\cV^GX)$, and let $(p_j)_{j\in J}\in \rmC_G(J\times X)$ satisfy $(p_j)_\cV = p$. Then following the proof of $(2)\Rightarrow (3)$ from Theorem~\ref{Thm:Fubini_Groups}, we obtain that for any $\delta > 0$, the set $\{j\in J: p_j\in \rmC_{\sigma_1+\delta}(X)\}$ is in $\cV$. When $X = \sa(G)$, Fact~\ref{Fact:Within_Delta_Lipschitz} allows us to find $q_j\in \rmC_{\sigma_1}(X)$ with $\|p_j-q_j\|\leq \delta/\sqrt{2}$, and this allows us to correct $(p_j)_{j\in J}$ to a new continuous extension of $p$ in $\rmC_{\sigma_1}(J\times X)$.

This section yields a large class of $G$-flows, the \emph{cofinally seminorm respecting} $G$-flows (Definition~\ref{Def:Seminorm_Respecting}) which satisfy the appropriate analog of Fact~\ref{Fact:Within_Delta_Lipschitz}, thus allowing the above proof to work. In particular, when $G$ is locally compact, all $G$-flows satisfy this property.

\subsection{The Gleason completion}
\label{Subsection:Gleason}

\begin{defin}
	\label{Def:Gleason_Complete}
	A $G$-space $X$ is called \emph{pre-Gleason} if whenever $A\in \op(X)$ and $x\in \ol{A}$, then for any $U\in \cN_G$, we have $x\in \Int(\ol{UA})$. A pre-Gleason $G$-flow is called \emph{Gleason complete}.
\end{defin}

To each $G$-space $X$, one can construct its \emph{Gleason completion}, a Gleason complete $G$-flow $\rmS_G(X)$ and a partially defined $G$-map from $\rmS_G(X)$ to $X$ satisfying a particular universal property. When $X$ is a $G$-flow, this map will be total, giving a factor map from $\rmS_G(X)$ to $X$. We take a moment to discuss the construction of the Gleason completion and the universal property that it satisfies, which we do in slighly more generality than in \cite{ZucDirectGPP}; we refer to \cite{ZucDirectGPP}, \cite{ZucMHP}, and \cite{LeBoudecTsankov} for more details.

\begin{defin}
	\label{Def:Near_Ult_Space}
	Fix a $G$-space $X$. A set $\cF\subseteq \op(X)$ has the \emph{near finite intersecton property}, or \emph{near FIP}, if whenever $Q\in [\cF]^{<\omega}$, we have $\bigcap_{A\in Q} UA\neq\emptyset$. A \emph{near ultrafilter} on $\op(X)$ is a set $\sfp\subseteq \op(X)$ which is maximal with respect to having the near FIP.  We let $\rmS_G(X)$ denote the set of near ultrafilters on $\op(X)$. If $A\in \op(X)$, we let $C_A = \{\sfp\in \rmS_G(X): A\in \sfp\}$ and $N_A = \rmS_G(X)\setminus C_A$. We equip $\rmS_G(X)$ with the compact Hausdorff topology given by the basis $\{N_A: A\in \op(X) \text{ and } \Int(X{\setminus}A)\neq\emptyset\}$ . Letting $G$ act on $\rmS_G(X)$ via $A\in g\sfp$ iff $g^{-1}A\in \sfp$, this action is continuous, making $\rmS_G(X)$ a $G$-flow (see \cite{ZucDirectGPP}). We call $\rmS_G(X)$ the \emph{Gleason completion} of $X$.
\end{defin} 

\begin{rem}
	Given $\sfp\in \rmS_G(X)$, a base of not-necessarily-open neighborhoods of $\sfp$ is given by $\{C_{UA}: A\in \sfp, U\in \cN_G\}$.
\end{rem}

We will soon see that $\rmS_G(X)$ is indeed Gleason complete (though one can also argue this directly). In \cite{ZucDirectGPP}, the universal property satisfied by $(\rmS_G(X), \pi_X)$ is only stated and proven when $X$ is a $G$-flow. However, one can phrase the universal property in an abstract way which works for any $G$-space $X$. 

\begin{defin}
	\label{Def:Irreducible}
	Let $X$ be a $G$-space, $Y$ a $G$-flow, and $\pi\subseteq Y\times X$. We call $(Y, \pi)$ an \emph{irreducible cover} of $X$ if the following all hold.
	\begin{itemize}
		\item 
		$\pi\subseteq Y\times X$ is closed and $G$-invariant.
		\item 
		$\pi$ is a partial function, and $\pi\colon \dom(\pi)\to X$ is a factor map.
		\item 
		For every $B\in \op(Y)$, there is some $x\in X$ with $\pi^{-1}(\{x\})\subseteq B$. In particular, this implies that $\dom(\pi)\subseteq Y$ is dense. 
	\end{itemize}
	When $X$ is a $G$-flow, the definition simplifies to stating that $\pi\colon Y\to X$ is a factor map satisfying the third bullet. In this case, we call $(Y, \pi)$ an \emph{irreducible extension} of $X$.
	
	If $(Y, \pi)$ is an irreducible cover of $X$, we define the \emph{fiber image} map $\pi_{fib}\colon \op(Y)\to \op(X)$ via $\pi_{fib}(B) = \{x\in X: \pi^{-1}(x)\subseteq B\}$. To see that $\pi_{fib}(B)$ is indeed open, suppose $x_i\to x$ with $x_i\not\in \pi_{fib}(B)$. We can find $y_i\in \dom(\pi)\setminus B$ with $(y_i, x_i)\in \pi$. We may assume $y_i\to y\in Y\setminus B$. But then $(y_i, x_i)\to (y, x)$, and as $\pi$ is closed, we have $(y, x)\in \pi$, implying $\pi^{-1}(\{x\})\not\subseteq B$.  
\end{defin}

\begin{defin}
	\label{Def:Univ_Prop_Gleason}
	Given a $G$-space $X$, define $\pi_X^G\subseteq \rmS_G(X)\times X$ by declaring that $(\sfp, x)\in \pi_X^G$ iff $\op(x, X)\subseteq \sfp$. 
\end{defin}

\begin{prop}
	\label{Prop:PI_X_Irreducible}
	$(\rmS_G(X), \pi_X^G)$ is an irreducible cover of $X$. Furthermore, we have that $(\pi_X^G)^{-1}$ is a function iff $X$ is pre-Gleason, in which case it is a continuous embedding. It follows that when $X$ is Gleason complete, $\pi_X^G\colon \rmS_G(X)\to X$ is an isomorphism.
\end{prop}

\begin{proof}
	Certainly $\pi_X^G$ is $G$-invariant. To check that it is closed, suppose $(\sfp_i, x_i)_{i\in I}$ is a net from $\pi_X^G$ and $(\sfp_i, x_i)\to (\sfp, x)\in \rmS_G(X)\times X$. Let $A\in \op(x, X)$. Then eventually $A\in \op(x_i, X)$, so $\sfp_i\in C_A$. As $\sfp_i\to \sfp$ and $C_A\subseteq \rmS_G(X)$ is closed, we have $\sfp\in C_A$, i.e.\ $A\in \sfp$.
	
	To see that $\pi_X^G$ is a partial function, suppose $x_0\neq x_1\in X$. As $X$ is a $G$-space, we can find $A_i\in \op(x_i, X)$ for $i< 2$ and some $U\in \cN(G)$ with $UA_0\cap UA_1 = \emptyset$. Hence $C_{A_0}\cap C_{A_1} = \emptyset$.

	To see that $\pi_X^G$ is a factor map, we check surjectivity and continuity. For each $x\in X$, $\op(x, X)\subseteq \op(X)$ has the near FIP, and any $p\in \rmS_G(X)$ with $\op(x, X)\subseteq \sfp$ satisfies $(\sfp, x)\in X$. Hence $\pi_X^G$ is onto. For continuity, let $(\sfp_i)_{i\in I}$ be a net from $\dom(\pi_X^G)$, and suppose $\sfp_i\to \sfp\in \dom(\pi_X^G)$. Fix $A\in \op(\pi_X^G(\sfp), X)$. Find some $B\in \op(\pi_X^G(\sfp), X)$ and $U\in\cN(G)$ with $UB\subseteq A$. As $B\in \sfp$, we have $\Int(X{\setminus}A)\not\in \sfp$. As non-membership is open in $\rmS_G(X)$, eventually $\Int(X{\setminus}A)\not\in \sfp_i$. For such $i\in I$, we must have $\pi_X^G(\sfp_i)\in \ol{A}$. 
	
	To check that $\pi_X^G$ is irreducible, fix $A\in \op(X)$ with $\Int(X{\setminus}A)\neq\emptyset$, and consider the basic open set $N_A\subseteq \rmS_G(X)$. Fix some $x\in \Int(X{\setminus}A)$. As $x\not\in \ol{A}$, we can find $B\in \op(x, X)$ and $U\in \cN(G)$ with $UA\cap UB = \emptyset$. This implies that $(\pi_X^G)^{-1}(\{x\})\subseteq N_A$.
	
	If $X$ is not pre-Gleason, find some $A\in \op(X)$, $x\in \ol{A}$, and $U\in \cN(G)$ with $x\not\in \Int(\ol{UA})$. Hence $X{\setminus}\ol{UA}$ is an open set with $x\in \ol{X{\setminus}\ol{UA}}$. Then $\op(x, X)\cup \{A\}$ and $\op(x, X)\cup \{X{\setminus}\ol{UA}\}$ both have the near FIP, so find $\sfp, \sfq\in \rmS_G(X)$ with $\op(x, X)\cup \{A\}\subseteq\sfp$ and $\op(x, X)\cup \{X{\setminus}\ol{UA}\}\subseteq \sfq$. Then $\sfp\neq \sfq$ and $(\sfp, x), (\sfq, x)\in \pi_X^G$. Hence $(\pi_X^G)^{-1}$ is not a function.
	
	If $X$ is pre-Gleason, then for any $x\in X$, we have $\{A\in \op(X): x\in \ol{A}\}\in \rmS_G(X)$. Since any $\sfp\in \rmS_G(X)$ satisfying $(\sfp, x)\in \pi_X^G$ must satisfy $\sfp\subseteq \{A\in \op(X): x\in \ol{A}\}$, we see that $(\pi_X^G)^{-1}$ is a function, with $(\pi_X^G)^{-1}(x) = \{A\in \op(X): x\in \ol{A}\}$. To see that $(\pi_X^G)^{-1}$ is a continuous embedding, it only remains to check continuity (as we have already verified the continuity of $\pi_X^G$). So let $(x_i)_{i\in I}$ be a net from $X$ with $x_i\to x\in X$. Fix some $A\in\op(X)$ with $(\pi_X^G)^{-1}(x)\in N_A$. Thus $x\not\in\ol{A}$. So eventually $x_i\not\in \ol{A}$, implying that eventually $(\pi_X^G)^{-1}(x_i)\in N_A$. 
\end{proof}

\begin{theorem}
	\label{Thm:Universal_Irreducible_Extension}
	$(\rmS_G(X), \pi_X^G)$ is the \emph{universal irreducible cover} of $X$, i.e.\ whenever $\pi\subseteq Y\times X$ is an irreducible cover, there is a $G$-map $\phi\colon \rmS_G(X)\to Y$ such that $(\phi\times \rm{id}_X)[\pi_X^G] \subseteq \pi$. 
\end{theorem}

\begin{proof}
	We define $\phi\colon \rmS_G(X)\to Y$ by declaring that $\phi(\sfp) = y$ iff for each $B\in \op(y, Y)$, we have $\pi_{fib}(B)\in \sfp$. The argument that this is well defined and satisfies the conclusion of the theorem statement is then very similar to the proof of Theorem~3.2 from \cite{ZucDirectGPP}.
\end{proof}

Notice that automatically, the map $\phi\colon \rmS_G(X)\to Y$ from Theorem~\ref{Thm:Universal_Irreducible_Extension} is an irreducible extension. Also note that when $X$ is a $G$-space, $(Y, \pi)$ is an irreducible cover of $X$, and $(Z, \psi)$ is an irreducible extension of $Y$, then writing 
$$\pi\circ \psi:= \{(z, x)\in Z\times X: \psi(z)\in \dom(\pi) \text{ and } \pi(\psi(z)) = x\},$$ we have that $(Z, \pi\circ \psi)$ is an irreducible cover of $X$. From this observation, the universal property of $(\rmS_G(X), \pi_X^G)$ implies that $\pi_{\rmS_G(X)}^G\colon \rmS_G(\rmS_G(X))\to \rmS_G(X)$ is an isomorphism. Thus by Proposition~\ref{Prop:PI_X_Irreducible}, $\rmS_G(X)$ is Gleason complete.

When $X$ is a $G$-flow, then all discussion of irreducible covers simplifies to discussing irreducible extensions, and we call $(\rmS_G(X), \pi_X^G)$ the \emph{universal irreducible extension} of $X$. This is the setting originally considered in \cite{ZucDirectGPP}. When $X$ is minimal and $(Y, \pi)$ is an irreducible extension, then $Y$ is also minimal, and furthermore, the map $\pi$ is \emph{highly proximal}, meaning that for any $x\in X$, there is a net $(g_i)_{i\in I}$ from $G$ such that $g\cdot \pi^{-1}(\{x\})$ converges in $2^X$ to a singleton. This is the notion originally considered in \cite{Auslander_Glasner}. Conversely, \emph{if $Y$ is minimal} and $\pi\colon Y\to X$ is highly proximal, then $\pi$ is irreducible. Thus \emph{among minimal flows}, $\pi_X^G\colon \rmS_G(X)\to X$ is the universal highly proximal extension, and \emph{among minimal flows}, $\rmS_G(X)$ is maximally highly proximal, explaining the ``MHP" terminology used in \cite{ZucMHP}. However, upon considering non-minimal flows, the notions of irreducible and highly proximal extensions become distinct. Thus with the authors of \cite{LeBoudecTsankov}, we have agreed upon the new terminology used both here and in \cite{LeBoudecTsankov}.    

When $X$ is pre-Gleason, the embedding $(\pi_X^G)^{-1}\colon X\to \rmS_G(X)$ given by Proposition~\ref{Prop:PI_X_Irreducible} satisfies a stronger universal property we have already encountered.

\begin{prop}
	For $X$ a pre-Gleason $G$-space, we have $(\rmS_G(X), (\pi_X^G)^{-1})\cong (\alpha_G(X), \iota_X^G)$.
\end{prop}

\begin{proof}
	Suppose $(Y, \phi)$ is a $G$-compactification of $X$. We define $\wt{\phi}\colon \rmS_G(X)\to Y$ by declaring that $\wt{\phi}(\sfp) = y$ iff for every $B\in \op(y, Y)$, we have $\phi^{-1}(B)\in \sfp$. It is routine to check that this works.
\end{proof}

In particular, whenever $\vec{X}:= \la X_i: i\in I\ra$ is a tuple of Gleason complete $G$-flows, then $\bigsqcup \vec{X}$ is a pre-Gleason $G$-space, thus allowing for a different construction of the ultracoproduct in this case. However, $\Sigma_\cU^GX_i$ need not be Gleason complete.

\begin{exa}
	\label{Exa:MHP_not_preserved}
	Suppose $G$ is discrete. Then Gleason completeness becomes an entirely topological property -- a $G$-flow $X$ is Gleason complete iff the space $X$ is extremally disconnected. By a theorem of Bankston \cite{Bankston_1984}, so long as $\cU$ is countably incomplete and $\cU$-many $X_i$ are infinite, then $\Sigma_\cU X_i$ is never basically disconnected, so in particular never extremally disconnected. In particular, the ultracoproduct of Gleason complete flows can fail to be Gleason complete.
\end{exa}

\subsection{Lower semi-continuous metrics via seminorms}

\begin{defin}
	\label{Def:SN_Function_On_Flow}
	Fix a $G$-flow $X$ and $\sigma\in \sn^1(G)$. We define $\partial_\sigma^X\colon X\times X\to [0, 1]$ by declaring that for any $x, y\in X$ and  $0\leq c< 1$, then
	$$\partial_\sigma^X(x, y)\leq c \Leftrightarrow \forall A\in \op(x, X)\, \forall B\in \op(y, Y)\, \forall \epsilon > 0\, [\rmB_{\sigma}(c+\epsilon)\cdot A\cap B\neq \emptyset].$$
	We define the function $\rho_\sigma^X\colon X\times X\to [0, 1]$ where given $x, y\in X$, we have
	$$\rho_\sigma^X(x, y) = \sup\{|f(x) - f(y)|: f\in \rmC_\sigma^1(X)\}.$$ 
	When $X$ is understood, we can omit it from the notation.
\end{defin}

Note that we have $\rho_\sigma \leq \partial_\sigma$. Furthermore, $\rho_\sigma$ is a pseudo-metric, and this pseudo-metric is \emph{lower semi-continuous}, i.e.\ for each $c\in [0, 1]$, the set $\{(x, y)\in X\times X: \rho_\sigma(x, y)\leq c\}$ is closed. If $\sigma$ is a norm, then $\rho_\sigma$ is a metric. Metrics of this form for $G$ a Polish group were first considered in the case of $X = \sa(G)$ in \cite{BYMT} and further investigated in \cite{ZucMHP}, where it is shown that for Gleason complete flows, $\rho_\sigma = \partial_\sigma$. To investigate these functions further, let us recall the following result of Ben Yaacov, a Tietze extension theorem for \emph{topometric spaces} that we will make frequent use of. Our statement is in part more general and in part less, but the proof carries over almost exactly.

\begin{fact}[\cite{Ben_Yaacov_2013}]
	\label{Fact:Topometric_Tietze}
	Let $X$ be a compact space and $\rho$ a lower semi-continuous pseudo-metric on $X$. If $Y\subseteq X$ is compact and $f\in \rmC(Y)$ is $\rho$-Lipschitz, then for any $c> 1$, there is $\tilde{f}\in \rmC(X)$ extending $f$ which is $c\rho$-Lipschitz.
\end{fact}   

\begin{cor}
	\label{Cor:Largest_PM_Below}
	For any $G$-flow $X$, $\rho_\sigma$ is the largest lower semi-continuous pseudo-metric on $X$ satisfying $\rho_\sigma\leq \partial_\sigma$.
\end{cor}

The techniques from the proof of Fact~\ref{Fact:Topometric_Tietze} yield the following analog of Fact~\ref{Fact:Within_Delta_Lipschitz}.

\begin{prop}
	\label{Prop:Within_Delta_Continuous_Lipschitz}
	Suppose $X$ is a compact space and $\rho$ is a lower semi-continuous pseudometric on $X$. If $\delta > 0$ and $f\in \rmC(X, \bbR)$ satisfies $\|f(x) - f(y)\|\leq \rho(x, y)+\delta$, then for any $c> 1$ and $\epsilon > 0$, there is a $c\rho$-Lipschitz $\tilde{f}\in \rmC(X, \bbR)$ with $\|\tilde{f} - f\|\leq \delta/2 + \epsilon$.
\end{prop}

\begin{proof}
	For each $\alpha\in \bbR$, set $F_\alpha = \{x\in X: f(x)\leq \alpha - \delta/2\}$ and $G_\alpha = \{x\in X: f(x)\geq \alpha + \delta/2\}$. Note that whenever $\alpha < \beta$, $x\in F_\alpha$, and $y\in G_\beta$, we have $|f(x) - f(y)|\geq \beta - \alpha + \delta$, implying that $\rho(x, y)\geq \beta - \alpha$. Then $(F_\alpha, G_\alpha)_{\alpha\in \bbR}$ is a \emph{Lipschitz system} (Definition~1.3 of \cite{Ben_Yaacov_2013}). By Lemma~1.5 of \cite{Ben_Yaacov_2013}, we can find for any finite $S\subseteq \bbR$ a $c\rho$-Lipschitz $\tilde{f}\in \rmC(X, \bbR)$ satisfying $\tilde{f}[F_\alpha]\leq \alpha$ and $\tilde{f}[G_\alpha]\geq \alpha$ for each $\alpha\in S$. By choosing $S$ appropriately, we can ensure $\|\tilde{f} - f\|\leq \delta/2 + \epsilon$.  
\end{proof}

\begin{defin}
	\label{Def:Seminorm_Respecting}
	Given a $G$-flow $X$ and $S\subseteq \sn^1(G)$, we say that $X$ is $S$-respecting if for each $\sigma\in S$, we have $\rho_\sigma = \partial_\sigma$. If $S = \{\sigma\}$, we write $\sigma$-respecting in place of $\{\sigma\}$-respecting. We say $X$ is \emph{cofinally seminorm respecting} if $X$ is $S$-respecting for some upwards cofinal $S\subseteq \sn^1(G)$. 
\end{defin}

We note that if $X$ is $\sigma$-respecting, then it is $c\sigma$-respecting for any $c> 1$.

\begin{prop}
	\label{Prop:Gleason_Respecting}
	Every Gleason complete $G$-flow is $\sn^1(G)$-respecting.
\end{prop}

\begin{proof}
	This is a rephrasing of the main results from Section~4 of \cite{ZucMHP}; while only stated there for norms on Polish groups, the general argument is almost identical.
\end{proof}

\begin{prop}
	\label{Prop:LC_Respecting}
	For locally compact $G$, every $G$-flow $X$ is $\snpc(G)$-respecting.
\end{prop}

\begin{proof}
	Given $x\in X$ and a compact $K\subseteq G$, then for any open $B\supseteq Kx$, we can find $A\in \op(x, X)$ with $KA\subseteq B$. It follows that
	given $\sigma\in \snpc(G)$ and $x, y\in X$, we have 
	\begin{align*}
		\partial_\sigma(x, y) =\begin{cases}
			\inf\{\sigma(g): gx = y\} \quad &\text{if $x, y$ belong to the same orbit},\\
			1 \quad &\text{otherwise}.
		\end{cases}
	\end{align*}
	It then follows from Corollary~\ref{Cor:Largest_PM_Below} that $\partial_\sigma = \rho_\sigma$. 	
\end{proof}

\begin{exa}
	Suppose $G$ is a non-Archimedean Polish group, and let $\vec{U} = \la U_n: n< \omega\ra$ be a base of clopen subgroups. If $\vec{c} = \la c_n: n< \omega\ra$ satisfies $c_n > 2c_{n+1}$, then a $G$-flow $X$ is $\sigma_{\vec{U}, \vec{c}}$-respecting iff for each $n< \omega$, the relation $R_n = \rm{cl}\{(x, gx): x\in X, g\in U_n\}$ (see Definition~\ref{Def:Orbit_Relation}) is an equivalence relation. For an example where this is not the case, consider $G = S_\infty$ and $X$ the space of ``least-$2$-forgetful" linear orders on $\bbN$, considered implicitly by Frasnay \cite{Frasnay} and more explicitly in unpublished work of Tsankov. This is the space $\rm{LO}(\bbN)/\sim_2$, where $L_0\sim_2 L_1$ iff $L_0 = L_1$ or if there are $m, n\in \bbN$ with $\{m, n\}\times( \bbN\setminus \{m, n\})\subseteq L_i$ for each $i< 2$ and $L_0$ and $L_1$ agree except on $\{m, n\}$. If $U_n$ denotes the pointwise stabilizer of $\{0,..., n-1\}$, then $R_2$ is not an equivalence relation. However, we note that $X$ is Fubini (see Corollary~\ref{Cor:UMF_CAP}). 
\end{exa}

\begin{theorem}
	\label{Thm:Seminorm_Respecting_Fubini}
	If $G$ is Fubini and $X$ is a cofinally seminorm respecting $G$-flow, then $X$ is Fubini.
\end{theorem}

\begin{proof}
	The proof is exactly as outlined in the beginning of the section. Given $\sigma_0\in \sn^1(G)$, find $\sigma_1\in \sn^1(G)$ as in item $(2)$ of Theorem~\ref{Thm:Fubini_Groups}. As $X$ is cofinally seminorm respecting, we may replace $\sigma_1$ by a larger seminorm if needed with $\rho_{\sigma_1} = \partial_{\sigma_1}$. Now fix an infinite set $J$ and $\cV\in \beta J$. We fix $c> 1$ and show that $\Sigma_\cV^GX\subseteq \alpha_G(J\times X)$ is $(\sigma_0, c\sigma_1)$-Fubini. Fix $p\in \rmC_{\sigma_0}(\Sigma_\cV^GX)$, and let $(p_j)_{j\in J}\in \rmC_G(J\times X)$ satisfy $(p_j)_\cV = p$. Following the proof of $(2)\Rightarrow (3)$ from Theorem~\ref{Thm:Fubini_Groups}, for any $\delta > 0$, we have $J_\delta:= \{j\in J: p_j\in \rmC_{\sigma_1+\delta}(X)\}\in \cV$. Since $\rho_{\sigma_1} = \partial_{\sigma_1}$, apply Proposition~\ref{Prop:Within_Delta_Continuous_Lipschitz} for each $j\in J_\delta$ to obtain $q_j\in \rmC_{c\sigma_1}(X)$ with $\|p_j - q_j\|\leq \delta$. We then mimic the rest of the proof of $(2)\Rightarrow (3)$ from Theorem~\ref{Thm:Fubini_Groups}.
\end{proof}

In the proof of the previous theorem, the constant $c> 1$ is not actually needed.

\begin{prop}
	\label{Prop:No_Constant}
	Suppose $I$ is an infinite set, $\la X_i: i\in I\ra$ are $G$-flows, and $\cU\in \beta I$, then if $\sigma_0, \sigma_1\in \sn(G)$ and the inclusion $\Sigma_\cU^GX_i\subseteq \alpha_G(\bigsqcup \vec{X})$ is $(\sigma_0, c\sigma_1)$-Tietze for every $c> 1$, then it is $(\sigma_0, \sigma_1)$-Tietze.
\end{prop}

\begin{proof}
	Fix $p\in \rmC_{\sigma_0}^1(\Sigma_\cU^GX_i)$, and for each $n\in \bbN$, we can find $(p_{i, n})_{i\in I}\in \rmC_{\sigma_1}^1(\bigsqcup_{i\in I}X_i)$ with $(p_{i, n})_\cU =\frac{(n-1)p}{n}$. In particular, we have for each $n\in \bbN$ that $I_n:= \{i\in I: \|p_{i, n} - p_{i, n+1}\|\leq 2^{-n}\|p\|\}\in \cV$. Set $I_\infty:= \bigcap_{n\in \bbN} I_n$. If $I_\infty\in \cU$, then for each $i\in I_\infty$, the $p_{i, n}$ converge uniformly to some $p_i\in \rmC_{\sigma_1}^1(X_i)$. For $i\in I\setminus I_\infty$, set $p_i\equiv 0$. Then $(p_i)_{i\in I}\in \rmC_{\sigma_1}^1(\bigsqcup_{i\in I}X_i)$ satisfies $(p_i)_\cU = p$. If $I_\infty\not\in \cU$, then for each $n\in \bbN$, the set $I_n' := I_n\setminus I_\infty$ is in $\cU$. For each $i\in (\bigcup_{k\leq n} I_k')\setminus I_{n+1}'$, define $p_i = p_{i, n}$, and for $i\in I\setminus I_0'$, set $p_i\equiv 0$. Then $(p_i)_{i\in I}\in \rmC_{\sigma_1}^1(\bigsqcup_{i\in I}X_i)$ and $(p_i)_\cU = p$. 
\end{proof}

\subsection{Weak Rigidity}
\label{Section:Rigidity}

\begin{defin}
	\label{Def:Weak_Rigid}
	We say that a $G$-flow $X$ is \emph{weakly rigid} if for every ultracopower of $X$, the ultracopower map $\pi_{X, \cU}^G\colon \Sigma_\cU^GX\to X$ is an isomorphism. 
\end{defin}

Note that any factor of a weakly rigid $G$-flow is also weakly rigid. Also note that every weakly rigid $G$-flow is Fubini. 

\begin{prop}
	\label{Prop:Weak_Rigid}
	If $X$ is a $G$-flow, then $X$ is weakly rigid iff for every $\sigma\in \sn(G)$, we have that $\rho_\sigma^X$ is continuous. 
\end{prop}

\begin{proof}
	First note that the continuity of $\rho_\sigma^X$ is equivalent to the statement that whenever $(x_i, y_i)_{i\in I}$ is a net from $X\times X$ with $(x_i, y_i)\to \Delta_X$, we have $\rho_\sigma^X(x_i, y_i)\to 0$. 
	
	Suppose $\rho_\sigma^X$ is not continuous, and let $(x_i, y_i)_{i\in I}$ be a net from $X\times X$ with $(x_i, y_i)\to \Delta_X$, but with $\rho_\sigma^X(x_i, y_i)$ bounded away from $0$. We may assume that $\lim x_i = \lim y_i = z$ for some $z\in X$. Let $\cU\in \beta I$ be any cofinal ultrafilter, and let $x_\cU  = \lim_{i\to \cU} (i, x_i) \in \Sigma_\cU^GX_i$, and similarly for $y_\cU$. Then $x_\cU\neq y_\cU$, but $\pi_{X, \cU}^G(x_\cU) = \pi_{X, \cU}^G(y_\cU) = z$.
	
	Now suppose for every $\sigma\in \sn(G)$ that $\rho_\sigma^X$ is continuous. This yields that for any set $I$ and $\cU\in \beta I$, the map $j_{X, \cU}^G\colon X\to \Sigma_\cU^GX$ is continuous. Continuity of every $\rho_\sigma^X$ also yields that the $G$-ultra\emph{power} $\Pi_\cU^GX$ coincides with $\im(j_{X, \cU}^G)$. As $\Pi_\cU^GX\subseteq \Sigma_\cU^GX$ is dense, we get equality, hence $j_{X, \cU}^G$ is an isomorphism.
\end{proof} 

This gives us a much simpler proof of the following result of Jahel-Zucker \cite{JahelZuckerExt} and Barto\v{s}ov\'a-Zucker \cite{ZucThesis}. Recall that every topological group admits a \emph{universal minimal flow}, a minimal flow which factors onto all other minimal flows, and that this flow is unique up to isomorphism. Let $\rmM(G)$ denote the universal minimal flow of $G$. We note that $\rmM(G)$ is \emph{coalescent}, i.e.\ every factor map from $\rmM(G)$ onto $\rmM(G)$ is an isomorphism (see \cite{Auslander}). It follows from this and the discussion after Theorem~\ref{Thm:Universal_Irreducible_Extension} that $\rmM(G)$ is Gleason complete. 

\begin{cor}
	\label{Cor:UMF_CAP}
	Let $G$ be Polish. Then $\rmM(G)$ is metrizable iff for every $G$-flow $Z$, the set $\rm{Min}_G(Z)\subseteq \rm{Sub}_G(Z)$ of minimal flows is Vietoris closed. In particular, this holds iff $\rmM(G)$ (and hence every minimal flow) is weakly rigid.
\end{cor}

\begin{proof}
	Let $\sigma\in \sn(G)$ be a \emph{norm}. By \cite{BYMT} and \cite{ZucMHP}, $\rmM(G)$ is metrizable iff $\partial_\sigma^{\rmM(G)} = \rho_\sigma^{\rmM(G)}$ is a compatible metric on $\rmM(G)$ iff $\partial_\sigma^{\rmM(G)}$ is continuous. Hence if $\rmM(G)$ is metrizable, then every ultracopower of $\rmM(G)$ is isomorphic to $\rmM(G)$, hence minimal. In particular, in $\rm{Sub}_G(Z)$, any Vietoris limit of minimal flows is a factor of an ultracoproduct of minimal flows (Proposition~\ref{Prop:Vietoris_Limit_WC}), hence a factor of an ultracopower of $\rmM(G)$, hence minimal. The converse follows directly from Proposition~\ref{Prop:Weak_Rigid} and \ref{Lem:Vietoris}. 
\end{proof}

Inspired by the above corollary, Basso and Zucker in \cite{BassoZucker} define a topologial group to be CAP (closed almost periodic) if the conclusion of the corollary holds. Hence among Polish groups, the CAP groups are exactly those with metrizable universal minimal flow. We note the following corollary for this more general class.

\begin{cor}
	\label{Cor:CAP_Ults}
	A topological group $G$ is CAP iff $\rmM(G)$ is weakly rigid iff the class of minimal $G$-flows is closed under ultracoproducts.
\end{cor}

\begin{proof}
	The only part which doesn't immediately follow from the above discussion is why, when $G$ is CAP, the ultracopower map onto $\rmM(G)$ must be an isomorphism. This is because $\rmM(G)$ is coalescent.
\end{proof}

We end by noting one more corollary, which follows from results implicit in \cite{ZucMHP} . If $G$ is a topological group and $H\leq G$ is a closed subgroup, the \emph{right uniformity} on $G/H$ is the uniformity whose entourages have the form $\{(gH, kH): UgH\cap kH\neq \emptyset\}$ for some $U\in \cN_G$. Let $\wh{G/H}$ denote the \emph{right completion} of $G/H$. We call $H\leq G$ \emph{co-precompact} if $\wh{G/H}$ is compact. 

\begin{cor}
	\label{Cor:G_over_H_wk_rigid}
	Suppose $G$ is Polish and that $X$ is Gleason complete and contains a point with dense orbit. Then $X$ is weakly rigid iff $X\cong \wh{G/H}$ for $H\leq G$ some closed, co-precompact subgroup. 
\end{cor}

\begin{proof}
	See Theorem~5.5 and Proposition~6.2 from \cite{ZucMHP}.
\end{proof}

\section{LRPC groups}
\label{Section:LRPC}

This section investigates Definition~\ref{Def:Seminorm_Respecting} in greater detail in the case that $G$ is LRPC. In particular, we will show that for $\sigma\in \snrpc(G)$, the class of $\sigma$-respecting $G$-flows is closed under ultracopowers. Along the way, we obtain a new characterization of RPC groups in terms of the Vietoris properties of the topologically transitive subflows of a $G$-flow.

\begin{lemma}
	\label{Lem:Rho_Partial_Z}
	Given any $\sigma\in \sn^1(G)$ and a tuple $\vec{X} = \la X_i: i\in I\ra$ of $\sigma$-respecting $G$-flows, then $\alpha_G(\bigsqcup \vec{X})$ is $\sigma$-respecting. 
\end{lemma}

\begin{proof}
	Write $Z := \alpha_G(\bigsqcup \vec{X})$. We always have $\rho_\sigma^Z\leq \partial_\sigma^Z$. For the other inequality, suppose $0< c < 1$ and $x, y\in Z$ satisfy $\partial_\sigma^Z(x, y)> c$. Let $A\in \op(x, Z)$, $B\in \op(y, Z)$, and $\epsilon > 0$ be such that $\rmB_{\sigma}(c+\epsilon)\cdot \ol{A}\cap \ol{B} = \emptyset$. For each $i\in I$, write $A_i = A\cap X_i$, $B_i = B\cap X_i$. Then $\rmB_{\sigma}(c+\epsilon)\cdot \ol{A_i}\cap \ol{B_i} = \emptyset$ for each $i\in I$. Since each $X_i$ is $\sigma$-respecting, use Fact~\ref{Fact:Topometric_Tietze} to find $p_i\in \rmC_{\sigma}^1(X_i)$ with $p_i|_{\ol{A_i}}\equiv 0$ and $p_i|_{\ol{B_i}} \equiv c+\epsilon/2$. It follows that $(p_i)_{i\in I}\in \rmC_\sigma^1(\bigsqcup \vec{X})$, and letting $p$ denote the continuous extension to $Z$, we have $p(x) = 0$ and $p(y)= c+\epsilon /2$.
\end{proof}

\begin{lemma}
	\label{Lem:LRPC_Tietze}
	Given $\sigma\in \snrpc(G)$, $\vec{X}= \la X_i: i\in I\ra$ a tuple of $\sigma$-respecting $G$-flows and $\cU\in \beta I$, then writing $Z = \alpha_G(\bigsqcup\vec{X})$ and $X = \Sigma_\cU^GX_i$, we have that $X\subseteq Z$ is $(\sigma, \sigma)$-Tietze. In particular, $\rho_\sigma^X = \rho_\sigma^Z|_{X^2}$. 
\end{lemma}

\begin{proof}
	By Proposition~\ref{Prop:Int_k_bounded_Fubini} and mimicking the proof of Proposition~\ref{Theorem:Seminorm_Respecting_Fubini}, one proves that $X\subseteq Z$ is $(\sigma, c\sigma)$-Tietze for any $c> 1$. Proposition~\ref{Prop:No_Constant} yields that $X\subseteq Z$ is $(\sigma, \sigma)$-Tietze. The last statement follows from this.
\end{proof}

By combining the previous two lemmas, and with notation as in Lemma~\ref{Lem:LRPC_Tietze}, it follows that to show that $X$ is $\sigma$-respecting, it suffices to show that $\partial_\sigma^X = \partial_\sigma^Z|_{X^2}$. In working towards this, we prove slightly more than we need, along the way obtaining a new characterization of RPC groups.

\begin{defin}
	\label{Def:Orbit_Relation}
	Let $X$ be a $G$-flow, and fix $U\in \cN(G)$. We define the $U$-relation $R_U^X := \rm{cl}\{(x, gx): x\in X, g\in U\} \subseteq X^2$. Equivalently, $(x, y)\in R_U^X$ iff $y\in \bigcap_{A\in \op(x, X)} \ol{UA}$. In the case $U = \rm{B}_\sigma(c)$ for some $\sigma\in \sn(G)$ and $c>0$, we can write $R_{\sigma, c}^X$ in place of $R_{\rmB_\sigma(c)}^X$. In particular, note that given $0\leq c< 1$, we have $\{(x, y)\in X^2: \partial_\sigma(x, y)\leq c\} = \bigcap_{\epsilon > 0} R_{\sigma, c+\epsilon}^X$. 
\end{defin}

We have the following characterization of when $U\in \cN(G)$ is RPC in terms of the behavior of $R_U^X$ as $X$ varies.

\begin{prop}
	\label{Prop:RPC_Vietoris_1}
	Given $U\in \cN(G)$, the following are equivalent.
	\begin{enumerate}
		\item 
		$U\in \ngrpc$ is RPC.
		\item
		Whenever $Z$ is a $G$-flow and  $(X_i)_{i\in I}$ is a net from $\rm{Sub}_G(Z)$ with $X_i\to X\in \rm{Sub}_G(Z)$, then $R_U^{X_i}\to R_U^X$ in $\exp(Z^2)$.
	\end{enumerate}
	
\end{prop}	

\begin{proof}
	$(1)\Rightarrow (2)$: By passing to a subnet if needed, we may assume $R_U^{X_i}\to S^X\subseteq X^2$. If $x\in X$, then we may find a subnet $(X_\alpha)_{\alpha}$ and $x_\alpha\in X_\alpha$ with $x_\alpha\to x$. Then for any $g\in U$, we have $(x_\alpha, gx_\alpha)\in R_U^{X_\alpha}$, and hence $(x, gx)\in S$. As $S^X$ is closed, we have $R_U^X\subseteq S^X$. This direction holds for any $U\in \cN(G)$.
	
	In the reverse direction, suppose $(x, y)\in S^X$, and fix $P_0\in \op(x, Z)$ and $Q_0\in \op(y, Z)$.  Setting $P = P_0\cap X$, $Q = Q_0\cap X$, we will show that $U{\cdot}\ol{P}\cap \ol{Q}\neq \emptyset$. Find $P_1\in \op(x, Z)$ and $Q_1\in \op(y, Z)$ with $\ol{P_1}\subseteq P_0$ and $\ol{Q_1}\subseteq Q_0$. We can find $V\subseteq \cN(G)$ with $V\subseteq U$ and with both $V{\cdot}\ol{P_1}\subseteq P_0$ and $V{\cdot}\ol{Q_1}\subseteq Q_0$.
	
	As $U\subseteq G$ is assumed to be RPC, and using Lemma~\ref{Lem:LRPC_1_bounded}, find a finite $F\subseteq U$ with $U\subseteq VFV$. We claim that given $Y\in \rm{Sub}_G(Z)$ with $R_U^Y\cap ((Y\cap P_1)\times (Y\cap Q_1))$ non-empty, we have that $F{\cdot}(P_0\cap Y)\cap (Q_0\cap Y)\neq \emptyset$. Towards a contradiction, suppose not. Then 
	\begin{align*}
		F{\cdot}(P_0\cap Y)\cap (Q_0\cap Y) &= \emptyset\\[1 mm]
		\Rightarrow FV{\cdot}(P_1\cap Y)\cap V{\cdot}(Q_1\cap Y) &= \emptyset\\[1 mm]
		\Leftrightarrow VFV{\cdot}(P_1\cap Y)\cap (Q_1\cap Y) &= \emptyset.
	\end{align*}
	This is a contradiction since $U\subseteq VFV$ and by our assumption on $Y$. 
	
	Eventually we have $R_U^{X_i}\cap ((P_1\cap X_i)\times (Q_1\cap X_i)) \neq \emptyset$. So eventually $F{\cdot}(P_0\cap X_i)\cap (Q_0\cap X_i) \neq \emptyset$. Passing to a subnet, this is witnessed by the same $g\in F\subseteq U$. If $w_i\in P_0\cap X_i$ is chosen so that $gw_i\in Q_0\cap X_i$, then passing to a subnet, if $w_i\to w\in \overline{P_0}\cap X$, we then have $gw\in \overline{Q_0}\cap X$ as desired.
	\vspace{3 mm}

	$\neg(1)\Rightarrow \neg(2):$ Fix $U\in \cN(G)\setminus \ngrpc$. We set $I := \cP_{fin}(U)$. Choose $\cU\in \beta I$ such that for each $F\in I$, $\{F'\in I: F\subseteq F'\}\in \cU$. We will show that the flows $Z := \alpha_G(I\times \sa(G))\cong \alpha_G(I\times G)$, $X_i := \{i\}\times \sa(G)$, and $X := \Sigma_\cU^G$ witness the failure of item $(2)$. We first observe that since $I\times G$ is a pre-Gleason $G$-space, we have $\alpha_G(I\times \Sa(G)) \cong \rmS_G(I\times G)$, and $\Sigma_\cU^G\Sa(G)$ can be identified with $$\{\sfp\in \rmS_G(I\times G): \forall S\in \cU\, [S\times G\in \sfp]\}.$$ 
	Let us write $x := \lim_{i\to \cU} (i, e_G)\in \Sigma_\cU^G\sa(G)$. We will find $g_i\in U$ such that, setting $y := \lim_{i\to \cU}g_ix_i$ that $(x, y)\not\in R_U^X$.  Fix $V\in \cN(G)$ which witnesses that $U$ is not RPC; by shrinking $V$ if needed, we can in fact assume that whenever $F\in I$, we have $A_F:= \Int(U\setminus VFV)\neq \emptyset$. Thus given $i\in I$, pick $g_i\in A_i$. To see that this works, fix $W\in \cN(G)$ with $W^2\subseteq V$. Set $A = \bigcup_{i\in I} A_i$, and observe that by the remark after Definition~\ref{Def:Near_Ult_Space}, $C_{I\times W}\cap X$ is a neighborhood of $x$ in $X$ and $C_{WA}\cap X$ is a neighborhood of $y$ in $X$. Fix $g\in U$. Whenever $i\in I$ satisfies $g\in i$, we have $VgV\cap A_i = \emptyset$. Thus for $\cU$-many $i\in I$, we have $gV\cap VA_i= \emptyset$. This implies $g\cdot (C_{I\times W}\cap X)\cap (C_{WA}\cap X) = (C_{I\times gW}\cap X)\cap (C_{WA}\cap X) = \emptyset$. Hence $(x, y)\not\in R_U^X$ as desired. 
\end{proof}

We record the following corollary which provides a new characterization of RPC groups. Recall that a $G$-flow $X$ is \emph{topologically transitive} if every open $G$-invariant subset of $X$ is dense. Write $\rm{TT}_G(X)\subseteq \rm{Sub}_G(X)$ for the set of subflows of $X$ which are topologically transitive. 

\begin{cor}
	\label{Cor:TT_RPC}
	A topological group $G$ is RPC iff for any $G$-flow $X$, $\rm{TT}_G(X)\subseteq \rm{Sub}_G(X)$ is Vietoris closed.
\end{cor}

\begin{proof}
	A $G$-flow $X$ is topologically transitive iff $R_G^X = X^2$. The corollary now follows from Proposition~\ref{Prop:RPC_Vietoris_1}
\end{proof}

\begin{prop}
	\label{Prop:RPC_Vietoris_2}
	Let $\vec{X} = \la X_i: i\in I\ra$ be $G$-flows, $Z = \alpha_G(\bigsqcup\vec{X})$, $X = \Sigma_\cU^GX_i$, and suppose $U\in \ngrpc$. Then $R_U^X = R_U^Z\cap X^2$. 
\end{prop}

\begin{proof}
	We always have $R_U^X\subseteq R_U^Z\cap X^2$. For the other direction, suppose $(x, y)\in R_U^Z\cap X^2$. Let $(x_j)_{j\in J}$ be a net from $Z$ and $(g_j)_{j\in J}$ be a net from $G$ with $x_j\to x$ and $g_jx_j\to y$. Let $Y_j = \ol{G\cdot x_j}$ and passing to a subnet if needed, let $Y = \lim_j Y_j$. Then by considering the natural map $\pi_I\colon Z\to \beta I$ and noting that each $\pi_I[Y_j]$ is a singleton, we see that $Y\subseteq X$, and by Proposition~\ref{Prop:RPC_Vietoris_1}, we have $R_U^{Y_j}\to R_U^Y$. Hence $(x, y)\in R_U^Y\subseteq R_U^X$.  
\end{proof}

\begin{cor}
	\label{Cor:SNRPC_Ults}
	With notation as in Lemma~\ref{Lem:LRPC_Tietze}, then $\partial_\sigma^X = \partial_\sigma^Z|_{X^2}$. In particular, whenever $\sigma\in \snrpc(G)$, the class of $\sigma$-respecting $G$-flows is closed under ultracoproducts.
\end{cor}

We end the section by isolating for LRPC groups a weaker and easier-to-verify condition than cofinally seminorm respecting which implies that a given $G$-flow is Fubini. 

\begin{defin}
	\label{Def:Urysohn}
	Given a $G$-flow $X$, $K$ and $L\in \exp(X)$, $\sigma\in \sn(G)$, and $0< c <1$, we say that $(K, L)$ is \emph{$(\sigma, c)$-separable} if  there is $p\in \rmC_\sigma(X, [0,1])$ with $p|_K\equiv 0$ and $p|_L\equiv c$. We say $(K, L)$ is \emph{$\sigma$-separable} if $(K, L)$ is $(\sigma, c)$-separable for every $0< c < 1$. 
	
	Given $U\in \cN_G$, and $\sigma\in \sn(G)$, we say that a $G$-flow $X$ is \emph{$(U, \sigma)$-Urysohn} if whenever $K, L\in \exp(X)$ satisfy $(K\times L)\cap R_U^X = \emptyset$, then $(K, L)$ is $\sigma$-separable. 
	
	If $G$ is LRPC, we say that a $G$-flow $X$ is \emph{Urysohn} if for any $U\in \ngrpc$, there is $\sigma\in \snrpc(G)$ such that $X$ is $(U, \sigma)$-Urysohn. If $\zeta\colon \ngrpc\to \snrpc(G)$ is a function, we call $X$ \emph{$\zeta$-Urysohn} if for each $U\in \ngrpc$, $X$ is $(U, \zeta(U))$-Urysohn.
\end{defin}

We make a few remarks about various aspects of the definition. First, for $K, L\in \exp(X)$ to be $(\sigma, c)$-separable, it suffices to find $p\in \rmC_\sigma(X, [0, 1])$ such that for some intervals $C, D\subseteq [0, 1]$ with $C< D$,and $\min(D)-\max(C)\geq c$, we have $p[K]\subseteq C$ and $p[L]\subseteq D$. Second, by a compactness argument, $(K\times L)\cap R_U^X = \emptyset$ iff there are open $A, B\in \op(X)$ with $K\subseteq A$, $L\subseteq B$, and $UA\cap B = \emptyset$. Third, if $\sigma\in \snrpc(G)$ and $X$ is a $\sigma$-respecting $G$-flow, then by Fact~\ref{Fact:Topometric_Tietze}, $X$ is $(\rmB_\sigma(1), \sigma)$-Urysohn. In particular, if $X$ is $S$-respecting for some $S\subseteq \snrpc(G)$ with the property that $\{\rmB_\sigma(1): \sigma\in S\}$ forms a base at $e_G$, then $X$ is Urysohn.

\begin{prop}
	\label{Prop:Urysohn_Fubini}
	If $G$ is LRPC and $X$ is an Urysohn $G$-flow, then $X$ is Fubini.
\end{prop}

\begin{proof}
	Fix infinite sets $I, J$ and ultrafilters $\cU\in \beta I$ and $\cV\in \beta J$. With $\psi\colon \fubini{\cU}{\cV}{G}X\to \Sigma_{\cU\otimes \cV}^G X$ as in Section~\ref{Section:Fubini}, we need to show that functions in $\im(\hat{\psi})$ separate points in $\fubini{\cU}{\cV}{G}X$. Let $x\neq y\in \fubini{\cU}{\cV}{G}X$. Write $W = \Sigma_\cV^GX$, $Y = \alpha_G(I\times W)$, and $Z = \alpha_G(J\times X)$. Using continuity of the action, find $A, B\in \op(Y)$ and $U\in \ngrpc$ with $x\in A$, $y\in B$, and $(\ol{A}\times \ol{B})\cap R_U^{Y}= \emptyset$. For each $i\in I$, let $A_i\in \op(W)$ satisfy $A\cap (\{i\}\times W) = \{i\}\times A_i$, and likewise for $B_i$. Note that $(\ol{A_i}\times \ol{B_i})\cap R_U^{W} = \emptyset$. By Proposition~\ref{Prop:RPC_Vietoris_2}, we have $(\ol{A_i}\times \ol{B_i})\cap R_U^{Z} = \emptyset$. Let $C_i, D_i\in \op(Z)$ satisfy $\ol{A_i}\subseteq C_i$, $\ol{B_i}\subseteq D_i$, and $(\ol{C_i}\times \ol{D_i})\cap R_U^{Z} = \emptyset$. For each $i\in I$ and $j\in J$, let $C_{ij}\in \op(X)$ satisfy $C_i\cap (\{j\}\times C_i) = \{j\}\times C_{ij}$, and likewise for $D_{ij}$. Note that $(\ol{D_{ij}}\times \ol{D_{ij}})\cap R_U^X = \emptyset$. As $X$ is Urysohn, there is $\sigma \in \snrpc(G)$ such that for each $(i, j)\in I\times J$, there is $p_{ij}\in \rmC_\sigma(X, [0, 1])$ with $p_{ij}|_{C_{ij}}\equiv 0$ and $p_{ij}|_{D_{ij}}\equiv 1/2$. Thus $((p_{ij})_{\cV})_\cU\in \im(\hat{\psi})$ and separates $x$ and $y$. 
\end{proof}

Using very similar proof ideas, we also have:

\begin{prop}
	\label{Prop:Urysohn_Ults}
	If $G$ is LRPC, $U\in \ngrpc$, and $\sigma\in \snrpc(G)$, then the class of $(U, \sigma)$-Urysohn $G$-flows is closed under ultracoproducts.
\end{prop}

\begin{proof}
	Let $\vec{X} = \la X_i: i\in I\ra$ be a tuple of $(U, \sigma)$-Urysohn $G$-flows. Write $X = \Sigma_\cU^GX_i$ and $Z = \alpha_G(\bigsqcup \vec{X})$. Let $K, L\in \exp(X)$ satisfy $(K\times L)\cap R_U^X  = \emptyset$. By Proposition~\ref{Prop:RPC_Vietoris_2}, also $(K\times L)\cap R_U^Z = \emptyset$. Find $A, B\in \op(Z)$ with $K\subseteq A$, $L\subseteq B$, and $(\ol{A}\times \ol{B})\cap R_U^Z = \emptyset$. Let $A_i = A\cap X_i$, $B_i = B\cap X_i$, and note that $(\ol{A_i}\times \ol{B_i})\cap R_U^{X_i} = \emptyset$. Fix $0< c < 1$, and let $p_i\in \rmC_\sigma(X_i, [0,1])$ satisfy $p_i|_{A_i} \equiv 0$ and $p_i|_{B_i} \equiv c$. Then $(p_i)_\cU\in \rmC_\sigma(\Sigma_\cU^GX_i)$ and satisfies $(p_i)_\cU|_K \equiv 0$ and $(p_i)_\cU|_L\equiv c$ as desired.
\end{proof}

\section{Weak types for flows of LRPC groups}

\subsection{Discrete groups}	

As a warmup, we first give an account of weak types in the case that $G$ is discrete. In this case (and also for the locally compact case dealt with later), weak types and the natural containment relation among them give precise characterizations of weak equivalence and weak containment. For actions of $\bbZ$ on Cantor space, we recover the notion of \emph{weak approximate conjugacy} introduced by Lin and Matui \cite{Lin_Matui}, and for free actions of a general countable group on Cantor space, we recover the notions of \emph{qualitative} weak containment/equivalence isolated in unpublished work of Elek \cite{Elek_Wk_Cont}. 

\begin{defin}
	\label{Def:Weak_Type_Discrete}
	Fix a discrete group $G$. Given $F\in \fin{G}$, we form the relational language $\cL_F:= \{E_g: g\in F\}\cup \{C_n: 1\leq n< \omega\}$, where each $E_g$ is binary and each $C_n$ is $n$-ary. 
	
	Given a $G$-flow $X$ and $\cO\in [\op(X)]^{<\omega}$, the \emph{full $\cL_F$-structure} on $\cO$, denoted $\str(\cO, F)$, is the $\cL_F$-structure with vertex set $\cO$ so that the following hold: 
	\begin{itemize}
		\item
		Given $g\in F$ and $A, B\in \cO$, we have $(A, B)\in E_g^{\str(\cO, F)}$ iff $\ol{gA}\cap \ol{B} = \emptyset$. 
		\item
		Given $A_0,...,A_{n-1}\subseteq \cO$, we have $(A_0,..., A_{n-1})\in C_n^{\str(\cO, F)}$ iff $\bigcup_{m<n} A_m = X$. 
	\end{itemize}
	Given a finite $\cL_F$-structure $\bM$, we say that $X$ \emph{realizes} $\bM$ if there is  $\cO\in [\op(X)]^{<\omega}$ of $X$ such that there is a bijective monomorphism  $e\colon \bM\to \str(\cO, F)$ (i.e.\ images of related tuples remain related). In this case we also say that $\cO$ \emph{realizes} $\bM$ and that $e$ is a \emph{realization} of $\bM$ (note that $\cO$ can realize several different $\bM$ up to isomorphism). The \emph{$F$-weak type} of $X$, denoted $\rm{tp}_F(X)$, is the collection of $\cL_F$-structures with vertex set some $n\in \bbN$ which are realized by $X$. If $H\subseteq G$ is infinite, the \emph{$H$-weak type} of $X$ is the set $\rm{tp}_H(X):= \bigcup \{\rm{tp}_F(X): F\in \fin{H}\}$. When $H = G$, we simply call $\rm{tp}_G(X)$ the \emph{weak type} of $X$.
\end{defin}

\begin{rem}
	One can consider the exact same definitions of full $\cL_{F}$-structure and of realization, but working with finite subsets of $\exp(X)$ instead of $\op(X)$. It is straightforward, using finitely many applications of normality in compact spaces, that if $X$ realizes $\bM$ using sets from $\exp(X)$, then $X$ realizes $\bM$ using sets from $\op(X)$. 
\end{rem}

Note that if $X$ is a factor of $Y$, then $\rm{tp}_G(X)\subseteq \rm{tp}_G(Y)$.

\begin{theorem}
	\label{Thm:Weak_Type_Discrete}
	Let $G$ be a discrete group, and let $X$ and $Y$ be $G$-flows. Then $X\preceq_G Y$ iff $\rm{tp}_G(X)\subseteq \rm{tp}_G(Y)$. In particular, $X\sim_G Y$ iff $\rm{tp}_G(X) = \rm{tp}_G(Y)$, so there is a \emph{set} of weak equivalence classes of $G$-flows (rather than a proper class).
\end{theorem}

\begin{proof}
	First assume $X\preceq_G Y$; we may assume $X = \Sigma_\cU Y$ for some $\cU\in \beta I$. Fix a finite $F\subseteq G$ and $\bM\in \rm{tp}_F(X)$, where $\bM$ has vertex set $n\in \bbN$ and is realized by $\cO:= \{A_m: m < n\}\in \fin{\op(X)}$ by the map $m\to A_m$. For each $m< n$, let $B_m\in \op(\beta(I\times Y))$ be such that $\ol{A_m}\subseteq B_m$ and $\ol{g\cdot B_k}\cap \ol{B_\ell} = \emptyset$ whenever $g\in F$ and $E_g^\bM(k, \ell)$ holds. Then for any $S\subseteq n$ with $X\subseteq \bigcup_{m\in S} A_m$, we have $X\subseteq \bigcup_{m\in S} B_m\in  \op(\beta(I\times Y))$, so for $\cU$-many $i\in I$, we have $\{i\}\times Y\subseteq \bigcup_{m\in S} B_m$. Thus for a suitable $i\in I$, it follows that $\bM\in \rm{tp}(Y)$ as realized by $\{B_m\cap (\{i\}\times Y): m< n\}$.

	In the other direction, suppose $\rm{tp}_G(X)\subseteq \rm{tp}_G(Y)$. We will find a suitable index set $I$, $\cU\in \beta I$, and a factor map $\phi\colon \Sigma_\cU Y\to X$. Let $\{(\cO_i, F_i): i\in I\}$ list all pairs where $\cO_i$ is a finite open cover of $X$ and $F_i\subseteq G$ is finite. We view $I$ as a directed set ordered under inclusion. Let $\cU\in \beta I$ be any cofinal ultrafilter. Let $\bM_i = \str(\cO_i, F_i)$. For each $i\in I$, let $\cQ_i$ be a finite open cover of $\{i\}\times Y$ which realizes $\bM_i$, and fix a realization $e_i\colon \bM_i\to \str(\cQ_i, F_i)$. It will be helpful to extend the domain of $e_i$ to all of $\op(X)$ by setting $e_i(A) = \emptyset$ whenever $A\not\in \bM_i$.

	We define $\phi\colon \Sigma_\cU Y\to X$ by declaring that $\phi(y) = x$ iff for every $A\in \op(x, X)$, we have $y\in \ol{\bigcup_{i\in I} e_i(A)}$. We check that $\phi$ is well defined. If $x_0, x_1\in X$, find $A_0\in \op(x_0, X)$ and $A_1\in \op(x_1, X)$ with $\ol{A_0}\cap \ol{A_1}= \emptyset$. Thus for $\cU$-many $i\in I$, we have $\ol{e_i(A_0)}\cap \ol{e_i(A_1)} = \emptyset$, so in particular, $\ol{\bigcup_{i\in I} e_i(A_0)}\cap \ol{\bigcup_{i\in I} e_i(A_1)} = \emptyset$. Thus $\phi(y)$, if it exists, is unique. To show $\phi(y)$ exists, suppose towards a contradiction that for each $x\in X$, there was $A_x\in \op(x, X)$ with $y\not\in \ol{\bigcup_{i\in I} e_i(A_x)}$. Let $\cO = \{A_{x_j}: j< k\}$ be a finite subcover of $X$. Then for $\cU$-many $i\in I$, we have $\{i\}\times Y\subseteq \bigcup_{j< k} e_i(A_{x_j})$. In particular, $\Sigma_\cU Y\subseteq \bigcup_{j< k} \ol{\bigcup_{i\in I} e_i(A_{x_j})}$, contradicting our assumption about $y$.

	To see that $\phi$ is continuous, fix a closed set $K\subseteq X$. We show that $$\phi^{-1}(K) = \bigcap_{\substack{A\in \op(X)\\ K\subseteq A}}\left(\ol{\bigcup_{i\in I} e_i(A)}\cap \Sigma_\cU Y\right).$$
	If $y\in \Sigma_\cU Y$ satisfies $\phi(y)\in K$, then clearly $y$ belongs to the right hand side. If $y\in \Sigma_\cU Y$ satisfies $\phi(y)\not\in K$, then for each $x\in K$, we can find $A_x\in \op(x, X)$ with $y\not\in \ol{\bigcup_{i\in I} e_i(A_x)}$. Passing to a finite subcover of $K$, we see that $y$ does not belong to the right hand side.

	To see that $\phi$ is onto, we simply note that by considering $K = \{x\}$, the above formula for $\phi^{-1}(K)$ is clearly non-empty.
	
	To see that $\phi$ is $G$-equivariant, fix $y\in \Sigma_\cU Y$ and $g\in G$. Write $x = \phi(y)$. Towards showing that $\phi(gy) = gx$, fix  $A\in \op(gx, X)$. Find $B, C\in \op(X)$ with $x\in B$, $A\cup C = X$, and $\ol{gB}\cap \ol{C} = \emptyset$. For $\cU$-many $i\in I$, we have $e_i(A)\cup e_i(C) = \{i\}\times Y$ and $\ol{g\cdot e_i(B)}\cap \ol{e_i(C)} = \emptyset$. In particular, for such $i\in I$, we have $g\cdot e_i(B)\subseteq e_i(A)$. As $y\in \ol{\bigcup_{i\in I} e_i(B)}$, we have $gy\in \ol{\bigcup_{i\in I} e_i(A)}$ as desired.
\end{proof}
	
Next, we investigate how weak types behave under taking ultracoproducts. This will allow us to equip the set of weak types with a canonical compact Hausdorff topology.

\begin{prop}
	\label{Prop:Weak_Type_Ults_Discrete}
	Fix a discrete group $G$, a tuple $\vec{X} = \la X_i: i\in I\ra$ of $G$-flows, and $\cU\in \beta I$. Then if $F\in \fin{G}$ and $\bM\in \rm{Str}(F)$, we have $\bM\in \rm{tp}_F(\Sigma_\cU X_i)$ iff for $\cU$-many $i\in I$, we have $\bM\in \rm{tp}_F(X_i)$.  
\end{prop}

\begin{proof}
	The proof of the forward direction is almost identical to the proof of the forward direction of Theorem~\ref{Thm:Weak_Type_Discrete}; if $\bM\in \rm{tp}_F(\Sigma_\cU X_i)$, then $\{i\in I: \bM\in \rm{tp}_F(X_i)\}\in \cU$.
	
	For the other direction, suppose we have $I_0:=\{i\in I: \bM\in \rm{tp}_F(X_i)\}\in \cU$. For each $i\in I_0$, let $e_i\colon \bM\to \str(\cO_i, F)$ be any realization, where $\cO_i\in \fin{\op(X_i)}$. For each $m< M$, set $B_m = \ol{\bigcup_{i\in I} e_i(m)}\cap \Sigma_\cU X_i.$ Towards showing that $\{B_m: m< M\}$ realizes $\bM$ (see the remark after Definition~\ref{Def:Weak_Type_Discrete}), we clearly have $\ol{gB_k}\cap \ol{B_\ell} = \emptyset$ whenever $g\in F$ and $(k, \ell)\in E_g^\bM$. Also, whenever $(a_0,..., a_{n-1})\in C_n^\bM$, we have $\bigcup_{j< n} B_{a_j} = \Sigma_\cU X_i$.
\end{proof}

Using Proposition~\ref{Prop:Weak_Type_Ults_Discrete}, we equip the set of weak equivalence classes of $G$-flows with a compact Hausdorff topology as follows. Letting $\rm{WT}(G) = \{\rm{tp}_G(X): X \text{ a $G$-flow}\}$, we simply view this as a subspace of $2^{\rm{Str}(G)}$ with the usual product topology. Proposition~\ref{Prop:Weak_Type_Ults_Discrete} then shows that ultracoproduct is a continuous operation on the space of weak types; in particular, $\rm{WT}(G)\subseteq 2^{\rm{Str}(G)}$ is a closed subspace. When $G$ is countable, this space is metrizable.

\subsection{LRPC groups}

We now work towards analogous results for LRPC groups. First, we modify Definition~\ref{Def:Weak_Type_Discrete} by strengthening the disjointness condition. We do this in two different ways, thus creating two different notions of weak type. 

\begin{defin}
	\label{Def:Weak_Type_LRPC}
	Fix an LRPC group $G$. Given $F\in \fin{G}$ and $N\in \fin{\ngrpc}$, we form the relational language $\cL_{F, N}:= \{E_{g, U}: g\in F, U\in N\}\cup \{C_n: 1\leq n< \omega\}$, where each $E_{g, U}$ is binary and each $C_n$ is $n$-ary. Write $\rm{Str}(F, N)$ for the set of finite $\cL_{F, N}$-structures with vertex set some $n\in \bbN$. If $H\subseteq G$ and $\cB\subseteq \ngrpc$, write $\rm{Str}(H, \cB) = \bigcup\{\rm{Str}(F, N): F\in \fin{H}, N\in \fin{\cB}\}$.
	
	Given a $G$-flow $X$ and $\cO\in [\op(X)]^{<\omega}$, the \emph{full $\cL_{F, N}$-structure} on $\cO$, denoted $\str(\cO, F, N)$, is the $\cL_{F, N}$-structure with vertex set $\cO$ so that the following hold: 
	\begin{itemize}
		\item
		Given $g\in F$ and $A, B\in \cO$, we have $(A, B)\in E_g^{\str(\cO, F, N)}$ iff $(\ol{gA}\times \ol{B}) = \emptyset$. 
		\item
		Given $A_0,...,A_{n-1}\subseteq \cO$, we have $(A_0,..., A_{n-1})\in C_n^{\str(\cO, F)}$ iff $\bigcup_{m<n} A_m = X$. 
	\end{itemize}
	
	Given $\bM\in \rm{Str}(F, N)$ and a $G$-flow $X$, we say that $X$ \emph{realizes} $\bM$ if for some $\cO\in \fin{\op(X)}$, there is a bijective monomorphism $e\colon \bM\to \str(\cO, F, N)$. We call $e$ a \emph{realization} of $\bM$ in $X$.

	Given $H\subseteq G$ and $\cB\subseteq \ngrpc$, the \emph{$(H, \cB)$-weak type} of $X$, denoted $\rm{tp}_{H, \cB}(X)$ is the set of $\bM\in \rm{Str}(H, \cB)$ which are realized by $X$. When $H = G$ and $\cB = \ngrpc$, we omit $\cB$ from the notation and call $\rm{tp}_G(X)$ the \emph{weak type} of $X$. 
	
	Given $G$-flows $X$ and $Y$, we say that  $X$ is \emph{weak type contained} in $Y$ if $\rm{tp}_G(X)\subseteq \rm{tp}_G(Y)$, and we say that $X$ and $Y$ are \emph{weak type equivalent} if $\rm{tp}_G(X) = \rm{tp}_G(Y)$. 
\end{defin}

\begin{rem}
	Similar to Definition~\ref{Def:Weak_Type_Discrete}, one can work with $\exp(X)$ instead of $\op(X)$.
\end{rem}

Note that every $G$-flow has a weak type and that weak type equivalence is an equivalence relation on the class of $G$-flows. We next show that for Urysohn $G$-flows, weak type interacts nicely with weak containment. We prove one direction in quite a bit more generality. To that end, suppose $\bM\in \rm{Str}(F, N)$ and $\zeta\colon N\to \ngrpc$ is a \emph{non-expansive map}, i.e.\ satisfying $\zeta(U)\subseteq U$ for each $U\in N$. We let $\zeta^*(\bM)\in \rm{Str}(F, \zeta[N])$ be defined on vertex set $M$ so that the $C_n$ relations are the same, and whenever $E_{g, U}^\bM(a, b)$ holds, then $E_{g, \zeta(U)}^{\zeta^*(\bM)}$ holds, and these are the only relations in $\zeta^*(\bM)$.

\begin{theorem}
	\label{Thm:Weak_Type_LRPC}
	Let $G$ be an LRPC group, and let $X$ and $Y$ be $G$-flows.
	\begin{enumerate}
		\item 
		If $X\preceq_G Y$, then $\rm{tp}_G(X)\subseteq \rm{tp}_G(Y)$. In particular, if $X\sim_G Y$, then $\rm{tp}_G(X) = \rm{tp}_G(Y)$.
		\item
		If $Y$ is Urysohn, $H\subseteq G$ is dense, $\cB\subseteq \ngrpc$ is a base at $e_G$, and $\zeta\colon \cB\to \cB$ is a non-expansive map such that $\zeta^*[\rm{tp}_{H, \cB}(X)]\subseteq \rm{tp}_{H, \cB}(Y)$, then $X\preceq_G Y$.
	\end{enumerate}
	In particular, if both $X$ and $Y$ are Urysohn, then $X\sim_G Y$ iff $\rm{tp}_{G}(X) = \rm{tp}_{G}(Y)$, and $\rm{tp}_G(Y)$ is completely determined by $\rm{tp}_{H, \cB}(Y)$. Thus  
\end{theorem}

\begin{proof}
	$(1)$: We may assume that for some infinite set $I$ and $\cU\in \beta I$, we have $X = \Sigma_\cU^G Y$. Fix $F\in \fin{G}$, $N\in \fin{\ngrpc}$, and $\bM\in \rm{tp}_{F, N}(X)$ is realized by $e\colon M\to \op(X)$. Note that when $(k, \ell)\in E_{g, U}^\bM$, we have $(\ol{g\cdot e(k)}\times \ol{e(\ell)})\cap R_U^{\alpha_G(I\times Y)} = \emptyset$ by Proposition~\ref{Prop:RPC_Vietoris_2}. For each $m< M$, let $B_m\in \op(\alpha_G(I\times Y))$ be such that $\ol{e(m)}\subseteq B_m$ and $(\ol{g\cdot B_k}\times \ol{B_\ell})\cap R_U^{\alpha_G(I\times Y)} = \emptyset$ whenever $g\in F$, $U\in N$, and $E_{g, U}^\bM(k, \ell)$ holds. Then for any $K\subseteq M$ with $X\subseteq \bigcup_{m\in K} e(m)$, we have $X\subseteq \bigcup_{m\in K} B_m\in  \op(\alpha_G(I\times Y))$, so for $\cU$-many $i\in I$, we have $\{i\}\times Y\subseteq \bigcup_{m\in K} B_m$. Thus for a suitable $i\in I$, it follows that $\bM\in \rm{tp}(Y)$ is realized by the map $m\to B_m\cap (\{i\}\times Y)$.
	\vspace{3 mm}
	
	$(2)$: We will find a suitable index set $I$, $\cU\in \beta I$, and a factor map $\phi\colon \Sigma_\cU^G Y\to X$. Let $\{(\cO_i, F_i, N_i): i\in I\}$ list all tuples where $\cO_i\in \fin{\op(X)}$,  $F_i\in [H]^{<\omega}$, and $N_i\in \fin{\cB}$. Let $\cU\in \beta I$ be any ultrafilter such that every $\cU$-large set is upwards cofinal. Let $\bM_i = \str(\cO_i, F_i, N_i)$. For each $i\in I$, let $\cQ_i$ be a finite open cover of $\{i\}\times Y$ which realizes $\zeta(\bM_i)$, and fix a realization $e_i\colon \zeta(\bM_i)\to \str(\cQ_i, F_i, N_i)$. It will be helpful to extend the domain of $e_i$ to all of $\op(X)$ by setting $e_i(A) = \emptyset$ whenever $A\not\in \bM_i$.

	We define $\phi\colon \Sigma_\cU^G Y\to X$ by declaring that $\phi(y) = x$ iff for every $A\in \op(x, X)$, we have $y\in \ol{\bigcup_{i\in I} e_i(A)}$. We check that $\phi$ is well defined. If $x_0, x_1\in X$, find $A_0\in \op(x_0, X)$, $A_1\in \op(x_1, X)$, and $U\in \ngrpc$ with $(\ol{A_0}\times \ol{A_1})\cap R_U^X = \emptyset$. Thus for $\cU$-many $i\in I$, we have $(\ol{e_i(A_0)}\times \ol{e_i(A_1)})\cap R_{\zeta(U)}^{\{i\}\times Y} = \emptyset$. As $Y$ is Urysohn, this implies $\ol{\bigcup_{i\in I} e_i(A_0)}\cap \ol{\bigcup_{i\in I} e_i(A_1)} = \emptyset$. Thus $\phi(y)$, if it exists, is unique  (this is the only part of the proof where we need that $Y$ is Urysohn). To show $\phi(y)$ exists, suppose towards a contradiction that for each $x\in X$, there was $A_x\in \op(x, X)$ with $y\not\in \ol{\bigcup_{i\in I} e_i(A_x)}$. Let $\cO = \{A_{x_j}: j< k\}$ be a finite subcover of $X$. Then for $\cU$-many $i\in I$, we have $\{i\}\times Y\subseteq \bigcup_{j< k} e_i(A_{x_j})$. In particular, $\Sigma_\cU^G Y\subseteq \bigcup_{j< k} \ol{\bigcup_{i\in I} e_i(A_{x_j})}$, contradicting our assumption about $y$.  
	
	To see that $\phi$ is continuous, fix a closed set $K\subseteq X$. We show that $$\phi^{-1}(K) = \bigcap_{\substack{A\in \op(X)\\ K\subseteq A}}\left(\ol{\bigcup_{i\in I} e_i(A)}\cap \Sigma_\cU^G Y\right).$$
	If $y\in \Sigma_\cU^G Y$ satisfies $\phi(y)\in K$, then clearly $y$ belongs to the right hand side. If $y\in \Sigma_\cU^G Y$ satisfies $\phi(y)\not\in K$, then for each $x\in K$, we can find $A_x\in \op(x, X)$ with $y\not\in \ol{\bigcup_{i\in I} e_i(A_x)}$. Passing to a finite subcover of $K$, we see that $y$ does not belong to the right hand side.

	To see that $\phi$ is onto, consider $K = \{x\}$ in the above formula for $\phi^{-1}(K)$. If $A, B\in \op(x, X)$ and $\ol{B}\subseteq A$, then find $C\in \op(X)$ and $U\in \cB$ with $A\cup C = X$ and $(\ol{B}\times \ol{C})\cap R_U^X = \emptyset$. It follows that for $\cU$-many $i\in I$, we have $e_i(A)\cup e_i(C) = \{i\}\times Y$ and $(\ol{e_i(B)}\times \ol{e_i(C)})\cap R_{\zeta(U)}^{\{i\}\times Y} = \emptyset$, in particular implying that $\ol{e_i(B)}\subseteq e_i(A)$. Hence the intersection in the formula for $\phi^{-1}(\{x\})$ is a directed intersection of compact sets, hence non-empty.  
	
	To see that $\phi$ is $G$-equivariant, it is enough to show that $\phi$ is $H$-equivariant, so fix $y\in \Sigma_\cU^G Y$ and $g\in H$. Write $x = \phi(y)$. Towards showing that $\phi(gy) = gx$, fix  $A\in \op(gx, X)$. Find $B, C\in \op(X)$ and $U\in \ngrpc$ with $x\in B$, $A\cup C = X$, and $(\ol{gB}\times \ol{C})\cap R_U^X = \emptyset$. For $\cU$-many $i\in I$, we have $e_i(A)\cup e_i(C) = \{i\}\times Y$ and $(\ol{g\cdot e_i(B)}\times \ol{e_i(C)})\cap R_{\zeta(U)}^{\{i\}\times Y} = \emptyset$. In particular, for such $i\in I$, we have $g\cdot e_i(B)\subseteq e_i(A)$. As $y\in \ol{\bigcup_{i\in I} e_i(B)}$, we have $gy\in \ol{\bigcup_{i\in I} e_i(A)}$ as desired.
\end{proof}

Given a map $\zeta\colon \ngrpc\to \snrpc(G)$ (in the context of Definition~\ref{Def:Urysohn}) and $c\in (0, 1)$, we define $\zeta_c\colon \ngrpc\to\ngrpc$ via $\zeta_c(U) = \rmB_{\zeta(U)}(c)$. Note that by modifying $\zeta$ if needed (by replacing $\zeta(U)$ by a pointwise larger seminorm), we can ensure that each $\zeta_c$ is non-expansive.

\begin{prop}
	\label{Prop:Weak_Type_Ults_LRPC}
	Fix an LRPC group $G$, a function $\zeta\colon \ngrpc\to \snrpc(G)$, a tuple $\vec{X} = \la X_i: i\in I\ra$ of $\zeta$-Urysohn $G$-flows, and $\cU\in \beta I$. Write $X = \Sigma_\cU^GX_i$ (which by Proposition~\ref{Prop:Urysohn_Ults} is $\zeta$-Urysohn).
	\begin{enumerate}
		\item 
		If $\bM\in \rm{tp}_G(X)$, then for $\cU$-many $i\in I$, we have $\bM\in \rm{tp}_G(X_i)$.
		\item 
		If $\bM\in \rm{tp}_G(X_i)$ for $\cU$-many $i\in I$, then for every $c\in (0, 1)$, we have $\zeta_c(\bM)\in \rm{tp}_G(X)$.
	\end{enumerate}
	Furthermore, upon endowing $2^{\rm{Str}(G, \ngrpc)}$ with the product topology and identifying subsets of $\rm{Str}(G, \ngrpc)$ with their characteristic functions, we have that $\rm{tp}_G(X)$ is uniquely defined by the property that 
	$$\zeta_{1/2}^*[\lim_\cU \rm{tp}(X_i)]\subseteq  \rm{tp}_G(X)\subseteq \lim_\cU \rm{tp}(X_i).$$
	The number $1/2$ here is unimportant; any $0< c< 1$ would do. 
\end{prop}

\begin{proof}
	The proof of $(1)$ is almost identical to the proof of  Theorem~\ref{Thm:Weak_Type_LRPC}(1).
	
	For the other direction, fix $F\in \fin{G}$ and $N\in \fin{\ngrpc}$, and suppose that $\bM\in \rm{Str}(F, N)$ satisfies $I_0:= \{i\in I: \bM\in \rm{tp}_{F, N}(X_i)\}\in \cU$. For $i\in I_0$, fix a realization $e_i\colon M\to \str(\cO_i, N, F)$, where $\cO_i\in \fin{\op(X)}$. For each $m< M$, write $B_m = \ol{\bigcup_{i\in I_0} e_i(m)}\cap X$. For any $k, \ell< M$ with $(k, \ell)\in E_{g, U}^\bM$, the assumption that each $X_i$ is $\zeta$-Urysohn implies that $(gB_k\times B_\ell)\cap R_{\zeta_c(U)}^{X} = \emptyset$ for any $c\in (0, 1)$. Also, if $(a_0,..., a_{n-1})\in C_n^\bM$, then $\bigcup_{j<n} B_{a_j} = X$.  By considering the remark after Definition~\ref{Def:Weak_Type_LRPC}, we have $\zeta_c^*(\bM)\in \rm{tp}_G(X)$.
	
	The ``furthermore" follows from Theorem~\ref{Thm:Weak_Type_LRPC}(2).
\end{proof}

Given an LRPC group $G$, write $\rm{WT}(G)\subseteq 2^{\rm{Str}(G, \ngrpc)}$ for the set of weak types of $G$-flows, and given a function $\zeta \colon \ngrpc\to \snrpc(G)$, write $\rm{WT}_\zeta(G)$ for the set of weak types of $\zeta$-Urysohn $G$-flows. With a bit more work, Proposition~\ref{Prop:Weak_Type_Ults_LRPC} will allow us to equip $\rm{WT}_\zeta(G)$ with a compact Hausdorff topology , which will be metrizable whenever $G$ is separable and metrizable. In particular, when $G$ is locally compact, there is a \emph{single} $\zeta$ such that \emph{every} $G$-flow is $\zeta$-Urysohn, thus giving us a compact Hausdorff topology on all of $\rm{WT}(G)$. While this topology is a direct result of considering the product topology on $2^{\rm{Str}(G, \ngrpc)}$, it is not just the subspace topology.

\begin{theorem}
	\label{Thm:Weak_Type_Topology_LRPC}
	Fix an LRPC group $G$ and a function $\zeta\colon \ngrpc\to \snrpc(G)$. There is a canonical compact Hausdorff topology $\tau$ on $\rm{WT}_\zeta(G)$, uniquely defined by the property that whenever $\la X_i: i\in I\ra$ is a tuple of $\zeta$-Urysohn $G$-flows and $\cU\in \beta I$, we have $\tau\-\lim_{\cU}\rm{tp}_G(X_i) = \rm{tp}_G(\Sigma_\cU^GX_i)$. 
	
	Furthermore, when $G$ is separable and metrizable, this topology on $\rm{WT}_\zeta(G)$ is metrizable. 
\end{theorem}

\begin{rem}
	The criterion on ultralimits tells us exactly which nets converge and what they converge to. However, one needs to check that this notion of convergence arises from a topological space. This can be done directly at the level of nets (see Exercise 11D of \cite{Willard}), but we give a more concrete description of the topology suitable for the ``furthermore." 
\end{rem}

\begin{proof}
	We reason a bit more abstractly (mainly to simplify notation). Let $S$ be a set and $f\colon S\to S$ a function (we will take $S = \rm{Str}(G, \ngrpc)$ and $f = \zeta_{1/2}^*$). We identify $\cP(S)$ and $2^S$, though we mostly work with the former. Let $W\subseteq \cP(S)$ (we will take $W = \rm{WT}_\zeta(G)$) satisfy the following:
	\begin{itemize}
		\item 
		For each $x\in W$, we have $f[x]\subseteq x$.
		\item
		For any $x, y\in W$, if $f^2[x]\subseteq y$ and $f^2[y]\subseteq x$, then $x = y$ (this holds by Theorem~\ref{Thm:Weak_Type_LRPC}).
		\item
		For each $y\in \ol{W}$ (where the closure is in the usual product topology), there is a unique $x\in W$ with $f[y]\subseteq x\subseteq y$ (this holds by Proposition~\ref{Prop:Weak_Type_Ults_LRPC}).
	\end{itemize}  
	We define a map $\pi\colon \ol{W}\to W$ where given $y\in \ol{W}$, $\pi(y)$ is the unique $x\in W$ as above. We will endow $W$ with the quotient topology induced by $\pi$; to show that this is compact Hausdorff, we need to check that the associated equivalence relation $E_\pi$ is closed in $\ol{W}\times \ol{W}$. First, we note that $(y, z)\in E_\pi$ iff both $f[y]\subseteq z$ and $f[z]\subseteq y$. Now let $(x_i, y_i)_{i\in I}$ be a net from $E_\pi$ with $x_i\to x$ and $y_i\to y$. By passing to a subnet, we may assume that the nets $(f[x_i])_{i\in I}$ and $(f[y_i])_{i\in I}$ are convergent, say with limits $u$ and $v$, respectively. Then we have $f[x]\subseteq u\subseteq y$, the first inclusion by the first bullet above, the second since $f[x_i]\subseteq y_i$ for every $i\in I$. Similarly, $f[y]\subseteq v\subseteq x$, showing that $(x, y)\in E_\pi$. The desired property of this topology on $\rm{WT}_\zeta(G)$ now follows  
	
	Now suppose additionally that $G$ is separable and metrizable. By Theorem~\ref{Thm:Weak_Type_LRPC}, we can identify $\rm{WT}_\zeta(G)$ with a subset of $2^{\rm{Str}(H, \cB)}$ for a countable dense $H\subseteq G$ and a countable base $\cB\subseteq \ngrpc$. We can treat $\zeta_{1/2}$ as a function with domain $\cB$, modifying it if needed (by replacing $\zeta_{1/2}(U)$ with a potentially smaller member of $\cB$) to have range $\cB$. Then the above considerations show that the topology on $\rm{WT}_\zeta(G)$ defined above is a Hausdorff continuous image of a compact metric space, hence is itself compact metric (\cite{Willard}, Corollary 23.2).
\end{proof}

\begin{exa}
	Let $G$ be Polish, and fix $\sigma\in \snpc(G)$ a \emph{norm}. In particular, if $G$ is locally compact, recall by Proposition~\ref{Prop:LC_Respecting} that every $G$-flow is $\sigma$-respecting. We discuss a slight variant of weak type which more naturally captures the topology on $\rm{WT}_\sigma(G)$, the space of weak types of $\sigma$-respecting $G$-flows. Given $F\in \fin{G}$, define the relational language $\cL_F'$ to contain the symbols $C_n$ as in Definitions~\ref{Def:Weak_Type_Discrete} and \ref{Def:Weak_Type_LRPC}, but the binary relations now have the form $\{E_{g, c}: g\in F, c\in (0, 1)\}$. Given a finite $\cL_F'$-structure $\bM'$ and $\bM\in \rm{Str}(G, \{\rmB_\sigma(c): 0< c< 1\})$, say that $\bM\ll \bM'$ if $M = M'$, the $C_n$ relations are the same, and the following both hold:
	\begin{itemize}
		\item
		For any $a, b\in M$ and $0< c'< 1$, we have that $(a, b)\in E_{g, c'}^{\bM'}$ implies that for some $c< c'$, we have $(a, b)\in E_{g, \rmB_\sigma(c)}^\bM$.
		\item
		For any $a, b\in M$ and $0< c< 1$, we have that $(a, b)\in E_{g, \rmB_\sigma(c)}^\bM$ implies that for some $c'> c$, we have $(a, b)\in E_{g, c'}^{\bM'}$.  
	\end{itemize}
	If $\bM'$ is a finite $\cL_F'$-structure and $X$ is a $G$-flow, say that $X$ \emph{realizes} $\bM'$ if for every $\bM\in \rm{Str}(G, \{\rmB_\sigma(c): 0< c< 1\})$ with $\bM\ll \bM'$, we have that $X$ realizes $\bM$ in the sense of Definiton~\ref{Def:Weak_Type_LRPC}.
	
	Write $\rm{Str}'(F)$ for the set of $\cL_F'$-structures with underlying set some $n\in \bbN$, and given $H\subseteq G$, write $\rm{Str}'(H) = \bigcup_{F\in \fin{H}} \rm{Str}'(F)$. Fix a countable dense subgroup $H\subseteq G$. Instead of viewing $\rm{Str}'(H)$ as just a set, we view it as a locally compact metric space with metric $\rho$ in a natural way, where given $0< d< 1$ and $\bM, \bN\in \rm{Str}'(G)$, we declare that $\rho(\bM, \bN)\leq d$ iff $M = N$, the $C_n$ relations are the same, and the following both hold:
		\begin{itemize}
		\item
		For any $a, b\in M$ and $0< c< 1$, we have that $(a, b)\in E_{g, c}^{\bM}$ implies that for some $c'\in (c-d, c+d)$, we have $(a, b)\in E_{g, c'}^\bN$.
		\item
		The above with $\bM$ and $\bN$ reversed. 
	\end{itemize}
	If the above doesn't hold for any $0< d< 1$, we declare that $\rho(\bM, \bN) = 1$. We then note that the set of members of $\rm{Str}'(G)$ realized by a given $G$-flow is closed in this metric space. Thus we can identify $\rm{WT}_\sigma(G)$ with a closed subspace of $\exp((\rm{Str}'(G), \rho))$.
\end{exa}

\section{Dynamical property (T)}

Property (T) is of vital importance to the study of representation theory and ergodic theory. For locally compact groups, it is equivalent to demanding that for representations, the property of not containing a non-zero invariant vector is closed in the space of weak types of unitary representations, and for p.m.p.\ actions, that the property of being ergodic is closed in the space of weak types of p.m.p.\ actions. It is thus natural to attempt to define a dynamical variant. However, various equivalent versions of Property (T) become inequivalent in the dynamical setting, so we propose two possible definitions.

\begin{defin}
	\label{Def:Dyn_T}
	We say that a $G$-flow $X$ is \emph{topologically ergodic} if there is no factor map from $X$ to a non-trivial motionless $G$-flow. We say that the topological group $G$ has \emph{Dynamical Property (T)} if for any $G$-flow $Z$, the subspace $\rm{TErg}_G(Z)\subseteq \rm{Sub}_G(Z)$ of topologically ergodic $G$-flows is closed; equivalently, if any ultracoproduct of topologically ergodic $G$-flows is topologically ergodic. We say that $G$ has \emph{weak Dynamical Property (T)} if whenever $X$ weakly contains a non-trivial motionless $G$-flow, then $X$ is not topologically ergodic. 
\end{defin}

While the direct sum of unitary representations which do not contain a non-zero invariant vector also enjoys this property (see Proposition~1.2.1 of \cite{B_dlH_V}), it is not necessarily true that a product of topologically ergodic $G$-flows remains topologically ergodic. For instance, the $\bbZ$-flow $X$ given by irrational rotation of the circle is minimal, but $X^2$ is not topologically ergodic. Hence, we isolate the two definitions above. 

\begin{prop}
	\label{Prop:Compact_T}
	Every compact group has Dynamical Property (T).
\end{prop}

\begin{proof}
	If $G$ is compact, then a $G$-flow $X$ is topologically ergodic iff $X$ is transitive. If $Z$ is a $G$-flow and $(X_i)_{i\in I}$ is a net from $\rm{TErg}_G(Z)$ with $X_i\to X\in \rm{Sub}_G(X)$, then it is routine using the compactness of $G$ to check that $X$ is also transitive.
\end{proof}

For positive results regarding Dynamical Property (T), the above proposition is about all there is.

\begin{prop}
	\label{Prop:Polish_Not_wk_T}
	Suppose $G$ is a topological group such that there is a topologically ergodic $G$-flow with two distinct fixed points. Then $G$ does not have Dynamical Property (T). In particular, no non-compact Polish group has Dynamical Property (T).
\end{prop}

\begin{proof}
	Let $X$ be the $G$-flow as in the proposition statement, with fixed points $x\neq y\in X$. Write $X = X_1$, and define the $G$-flow $X_n$ by gluing together $n$ copies of $X$ in a line. More precisely, let $X_{1, n}$,..., $X_{n, n}$ be the $n$ copies of $X$, and write $x_{m, n}$, $y_{m, n}$ be the points in $X_{m, n}$ corresponding to $x$ and $y$. We form $X_n$ by attaching $y_{k, n}$ to $x_{k+1, n}$ for each $1\leq k< n$. On each $X_n$, define the function $\psi_n\colon X_n\to [0, 1]$ by setting $\psi_n(z) = k/n$ iff $z\in X_{k, n}\setminus \{y_{k, n}\}$ for $k< n$, and setting $\psi_n(z) = 1$ for any $z\in X_{n, n}$. Let $\psi$ denote the disjoint union of the functions $\psi_n$. Define a space $Z$ which as a set is $(\bigsqcup_{1\leq n< \omega} X_n)\sqcup [0, 1]$; we define a compact Hausdorff topology on $Z$ by declaring that each $X_n$ is a clopen subspace, and given a net $(z_i)_{i\in I}$ from $Z$ which isn't eventually contained in some $X_n$, we declare that $\lim_{i\in I} z_i = c\in [0, 1]$ iff $\lim_{i\in I} \psi(z_i) = c$. We turn $Z$ into a $G$-flow by viewing each $X_n$ as a subflow and declaring that the action is trivial on $[0, 1]$. Then we have that each $X_n$ is topologically ergodic, but $\lim X_n$ is the motionless subflow $[0, 1]$.
	
	For the ``in particular," we simply note that for any non-compact Polish group, $\sa(G)$ contains infinitely many minimal subflows (see for instance \cite{Bartosova_Thesis}). Let $M\neq N\subseteq \sa(G)$ be two distinct subflows, and let $X$ be the quotient of $\sa(G)$ which collapses both $M$ and $N$ down to fixed points. 
\end{proof}

It is an open question whether, given a general non-precompact group $G$, $\sa(G)$ contains two distinct minimal subflows. This is related to the concept of \emph{ambitability} from \cite{Pachl}.

We turn now to weak Dynamical Property (T), again with some negative results. 

\begin{prop}
	\label{Prop:LC_Not_T}
	No locally compact, non-compact group has weak Dynamical Property (T). 
\end{prop}

\begin{proof}
	We construct an ultracopower of $\sa(G)$ which admits a non-trivial, $G$-invariant continuous function. Fix $I = \rm{FS}(G)\times (0, 1)$, and let $\cU\in \beta I$ be any ultrafilter such that for any $F\in \rm{FS}(G)$ and any $\epsilon > 0$, we have $\{(F', \epsilon'): F'\supseteq F, \epsilon' < \epsilon \}\in \cU$ and furthermore, so that for any $\delta > 0$ and $n< \omega$, we have $\{(F, \epsilon): |F^n|\cdot \epsilon < \delta\}\in \cU$. Fix some $\sigma = \sigma_{\vec{U}}\in \snpc^1(G)$ for some $\vec{U}$ with $U_{n+1}^3\subseteq U_n$ (Fact~\ref{Fact:Birkhoff_Kakutani}). In particular, note that the Haar measure of $U_n$ tends to zero. For each $i= (F, \epsilon) \in I$, set $p_i = \Phi(\sigma, 0, F, \epsilon )$ (Notation~\ref{Notation:Fubini_SN}). Then $\|p_i\|\geq 1/2$ for $\cU$-many $i\in I$ (indeed, our demand on $\cU$ ensures that we can bound the Haar measure of the set $\{g\in G: p_i(g)\geq 1/2\}$ away from $0$ for $\cU$-many $i\in I$) and $(p_i)_\cU$ is $G$-invariant.    
\end{proof}

It is natural to ask if there are examples of non-precompact topological groups which do have weak Dynamical Property (T). Given Corollary~\ref{Cor:TT_RPC} and the recent result of Ibarluc\'ia \cite{Ibarlucia_Prop_T} that all Polish RPC groups have Property (T), perhaps it is true that \emph{every} RPC group has weak Dynamical Property (T).

	\bibliographystyle{amsplain}
	\bibliography{bibzucker}
	
\end{document}